\documentclass[11pt,righttag]{amsart}
\usepackage{diagrams}
\usepackage{graphicx}

\begin{document}

\topmargin-0.1in
\textheight8.5in
\textwidth5.5in
\footskip35pt
\oddsidemargin.5in
\evensidemargin.5in

\newcommand{\V}{{\cal V}}      
\renewcommand{\O}{{\cal O}}
\newcommand{\LL}{\cal L}
\newcommand{\Ext}{\hbox{{\rm Ext}}}
\newcommand{\Tor}{\hbox{Tor}}
\newcommand{\Hom}{\hbox{Hom}}
\newcommand{\Proj}{\hbox{Proj}}
\newcommand{\GrMod}{\hbox{GrMod}}
\newcommand{\grmod}{\hbox{gr-mod}}
\newcommand{\tors}{\hbox{tors}}
\newcommand{\rank}{\hbox{rank}}
\newcommand{\End}{\hbox{End}}
\newcommand{\GKdim}{\hbox{GKdim}}
\newcommand{\im}{\hbox{im}}
\renewcommand{\ker}{\hbox{ker}}
\newcommand{\isom}{\cong}
\newcommand{\lk}{\text{link}}
\newcommand{\soc}{\text{soc}}

\newcommand{\lonto}{{\protect \longrightarrow\!\!\!\!\!\!\!\!\longrightarrow}}

\newcommand{\m}{{\mu}}
\newcommand{\gl}{{\frak g}{\frak l}}
\newcommand{\ssl}{{\frak s}{\frak l}}
\newcommand{\I}{\mathfrak{I}}

\newcommand{\ds}{\displaystyle}
\newcommand{\s}{\sigma}
\renewcommand{\l}{\lambda}
\renewcommand{\a}{\alpha}
\renewcommand{\b}{\beta}
\newcommand{\G}{\Gamma}
\newcommand{\g}{\gamma}
\newcommand{\z}{\zeta}
\newcommand{\e}{\epsilon}
\newcommand{\D}{\Delta}
\renewcommand{\d}{\delta}
\newcommand{\p}{\rho}
\renewcommand{\t}{\tau}

\newcommand{\C}{{\mathbb C}}
\newcommand{\N}{{\mathbb N}}
\newcommand{\Z}{{\mathbb Z}}
\newcommand{\ZZ}{{\mathbb Z}}
\newcommand{\K}{{\mathcal K}}
\newcommand{\F}{{\mathcal F}}
\newcommand{\cS}{\mathcal S}

\newcommand{\rowxy}{(x\ y)}
\newcommand{\colxy}{ \left({\begin{array}{c} x \\ y \end{array}}\right)}
\newcommand{\scolxy}{\left({\begin{smallmatrix} x \\ y
\end{smallmatrix}}\right)}

\renewcommand{\P}{{\Bbb P}}

\newcommand{\la}{\langle}
\newcommand{\ra}{\rangle}
\newcommand{\tensor}{\otimes}

\newtheorem{thm}{Theorem}[section]
\newtheorem{lemma}[thm]{Lemma}
\newtheorem{cor}[thm]{Corollary}
\newtheorem{prop}[thm]{Proposition}

\theoremstyle{definition}
\newtheorem{defn}[thm]{Definition}
\newtheorem{notn}[thm]{Notation}
\newtheorem{ex}[thm]{Example}
\newtheorem{rmk}[thm]{Remark}
\newtheorem{rmks}[thm]{Remarks}
\newtheorem{note}[thm]{Note}
\newtheorem{example}[thm]{Example}
\newtheorem{problem}[thm]{Problem}
\newtheorem{ques}[thm]{Question}
\newtheorem{conj}[thm]{Conjecture}
\newtheorem{thingy}[thm]{}

\newcommand{\onto}{{\protect \rightarrow\!\!\!\!\!\rightarrow}}
\newcommand{\donto}{\put(0,-2){$|$}\put(-1.3,-12){$\downarrow$}{\put(-1.3,-14.5) 

{$\downarrow$}}}

\newcounter{letter}
\renewcommand{\theletter}{\rom{(}\alph{letter}\rom{)}}

\newenvironment{lcase}{\begin{list}{~~~~\theletter} {\usecounter{letter}
\setlength{\labelwidth4ex}{\leftmargin6ex}}}{\end{list}}

\newcounter{rnum}
\renewcommand{\thernum}{\rom{(}\roman{rnum}\rom{)}}

\newenvironment{lnum}{\begin{list}{~~~~\thernum}{\usecounter{rnum}
\setlength{\labelwidth4ex}{\leftmargin6ex}}}{\end{list}}



\title{$\K_2$ Factors of Koszul Algebras and Applications to  Face Rings}

\keywords{$\K_2$ algebra, Koszul algebra, Yoneda algebra, Stanely-Reisner ring, Face ring}

\author[  Conner, Shelton ]{ }
\maketitle

\begin{center}

\vskip-.2in Andrew Conner \\
\bigskip

Department of Mathematics\\
Wake Forest University\\
Winston-Salem, NC 27109\\
\bigskip

 Brad Shelton \\
\bigskip

Department of Mathematics\\ University of Oregon\\
Eugene, Oregon 97401
\\ \ \\

\end{center}

\setcounter{page}{1}

\thispagestyle{empty}

\vspace{0.2in}

\begin{abstract}
\baselineskip15pt

Generalizing the notion of a Koszul algebra, a graded $k$-algebra $A$ is $\K_2$ if its Yoneda algebra $\Ext_A(k,k)$ is generated as an algebra in cohomology degrees 1 and 2.  We prove a strong theorem about $\K_2$ factor algebras of Koszul algebras and use that theorem to show  the Stanley-Reisner face ring of a simplicial complex $\D$ is $\K_2$ whenever the Alexander dual simplicial complex $\D^*$ is (sequentially) Cohen-Macaulay.

\end{abstract}

\bigskip

\baselineskip18pt


\section{Introduction}

Let $k$ be a field.  Throughout this paper we use the phrase \emph{graded $k$-algebra} to mean a connected, $\N$-graded, locally finite-dimensional $k$-algebra which is finitely generated in degree 1.  Let $A$ be a graded $k$-algebra. We use the term \emph{(left or right) ideal} to mean a graded (left or right) ideal of $A$ generated by homogeneous elements of degree at least 2. The augmentation ideal is $A_+=\bigoplus_{i\ge 1} A_i$. We abuse notation and use $k$ (or $_Ak$ or $k_A$) to denote the trivial graded $A$-module $A/A_+$. The bigraded Yoneda algeba of $A$ is  the $k$-algebra $E(A) = \bigoplus_{i,j\ge 0} E^{i,j}(A) = \bigoplus_{i,j\ge 0} \Ext^{i,j}_A(k,k)$. (Here $i$ denotes the cohomology degree and $j$ denotes the internal degree inherited from the grading on $A$.) Let $E^p(A)=\bigoplus_{q} E^{p,q}(A)$.

A graded $k$-algebra $A$ is called Koszul (\cite{Priddy}) if $E(A) $ is generated by $E^1(A) = E^{1,1}(A)$ as a $k$-algebra. Equivalently, $A$ is Koszul if it satisfies the strong purity condition $E^{i,j}(A) = 0$ for all $i\ne j$. (Purity is the homological condition that for each $i$, $E^i(A)$ is supported in at most one internal degree.) Koszul algebras play an important role in many branches of mathematics, but Koszul algebras must be quadratic algebras.  

In \cite{CS}, Cassidy and Shelton introduced a generalization of Koszul that includes  non-quadratic algebras and algebras with relations in more than one degree.  Following \cite{CS}, we say graded $k$-algebra $A$ is $\K_2$ if $E(A)$ is generated as a $k$-algebra by $E^1(A)$ and $E^2(A)$.  As $E^2(A)$ encodes the defining relations of $A$, this is essentially the smallest space that could generate $E(A)$ when $A$ has non-quadratic relations. (Note that $E^{2,2}(A)$ is always generated by $E^1((A)$.)  Clearly Koszul algebras are $\K_2$, and in \cite{G} it was shown that Berger's $D$-Koszul algebras are $\K_2$ as well. In \cite{CS} it was shown that class of $\K_2$ algebras is closed under tensor products, regular central extensions and even graded Ore extensions. Moreover,  it was shown in \cite{CS} that the class of $\K_2$ algebras includes Artin-Schelter regular algebras of global dimension 4 on three linear generators and graded complete intersections.  That paper also classifies, via a simple combinatorial algorithm, the $\K_2$ (non-commutative) monomial algebras.

The authors' current project began with the wish to understand the $\K_2$ property for commutative monomial algebras.  We concentrated our attention on face rings.

Given any simplicial complex $\D$ on vertex set $[n]=\{1, \ldots, n\}$, the Stanley Reisner ideal of $\D$ is the square-free monomial ideal $I_\D$ in 
$S = k[x_1,\ldots, x_n]$ generated by the monomials $\Pi_{i\in \tau} x_i$ for $\tau\subset [n]$, $\tau\notin\D$. The factor algebra $k[\D]=k[x_1,\ldots,x_n]/I_{\D}$ is called the Stanley-Reisner ring, or \emph{face ring}, of the simplicial complex $\D$.

The rich interplay between combinatorics, topology and the properties of the face ring $k[\D]$ are well known (see for example \cite{MS} and \cite{Stanley}), but little seems to be known about the algebraic structure of $E(k[\D])$. At least the Koszul case is well-understood via a theorem of  Fr\"{o}berg (\cite{Fro}): If $I_\D$ is generated by quadratic monomials then $k[\D]$ is Koszul.  Of course $I_\D$ is generated by quadratic monomials if and only if the minimal missing faces of $\D$ consist solely of edges. 

It seems natural to look for combinatorial or geometric conditions on $\D$ that would ensure the algebra $k[\D]$ is $\K_2$.  We will show that the problem is considerably more subtle than the Koszul case. 

Although the main theorem of this paper was heavily influenced by our study of face rings, the theorem itself is not  about face rings at all.  We require one additional definition before stating the theorem.

\begin{defn}
Let $A$ be a graded $k$-algebra and let $M$ be a graded left $A$-module. Let  $D_1(A)$ be the subalgebra of $E(A)$ generated by $E^1(A)$ and let $D_2(A)$ be the subalgebra of $E(A)$ generated by $E^1(A)+E^2(A)$.  For $n=1, 2$, we call $M$ a $\K_n$ $A$-module if $\Ext_A(M,k)$ is generated as a left $D_n(A)$-module by $\Ext_A^0(M,k)$. We call $M$ a Koszul $A$-module if $M$ is $\K_1$ and there exists $d\in\Z$ such that $M=AM_d$.
\end{defn}

We note that $A$ is a $\K_2$ algebra if and only if $_Ak$ is a $\K_2$ $A$-module and  $A$ is a $\K_1$ algebra in the sense of \cite{PhanThesis} if and only if $_Ak$ is a $\K_1$ $A$-module. We also note that $M$ is a Koszul module if and only if there exists a $d\in\Z$ such that $M(d)$ is a Koszul module in the sense of \cite{PP}.
(Obviously, one could define the notion of a $\K_n$ algebra or module for any $n>2$, however we see no motivation for doing so at this time.)

Our main theorem is  the following extension of an important theorem on Koszul algebras due to Backelin and Fr\"{o}berg \cite{BackFro}:

 \begin{thm} 
  \label{factorThmIntro}
  Let $A$ be a Koszul algebra and $\I\subset A$ an ideal. Assume $B=A/\I$ acts trivially on $\Ext_A(B,k)$.
 If $_A\I$ is a $\K_2$ $A$-module, then $B=A/\I$ is a $\K_2$ algebra.
 \end{thm}
 
 The technical hypothesis that $B$ acts trivially on $\Ext_A(B,k)$ is satisfied when $A$ is commutative or graded-commutative, making Theorem \ref{factorThmIntro} particularly useful in those settings.  
 
Theorem \ref{factorThmIntro} is proved in Section \ref{factorK2} as Theorem \ref{factorThm}.  Section \ref{Notation} deals with technical notation and Section \ref{modules} gives technical information on Koszul and $\K_2$-modules. Section \ref{CLmods} contains several useful lemmas on the closely related notion of modules having a componentwise linear resolution.  We return to face rings in Section \ref{faceRings} where we show that a well-known class of face rings are indeed $\K_2$.

\begin{thm}
\label{faceRingsIntro}
Let $\D$ be a simplicial complex and $\D^*$ its Alexander dual. If $\D^*$ is (sequentially) Cohen-Macaulay over $k$, then $k[\D]$ is a $\K_2$ algebra.
\end{thm}

In the final two sections of the paper, we discuss several illustrative examples of face rings.  As well as exhibiting many subtleties, these examples show that the sufficient conditions of Theorems \ref{factorThmIntro} and \ref{faceRingsIntro} are not neccessary.  They also proved (negative) answers to some questions from \cite{CS}.

\section{Notation}
\label{Notation}

 Let $A$ be a graded $k$-algebra. We denote the category of locally finite dimensional, $\Z$-graded left (resp.\ right) $A$-modules with degree 0 homomorphisms by $\textbf{Gr-A-mod}^{f}$ (resp.\ $\textbf{Gr-mod-A}^{f}$). The term \emph{graded left (resp.\ right) $A$-module} will refer to an object of the appropriate module category. We do not assume that modules are finitely generated. 
 
If $M$ is a graded $A$-module, we write $M(d)$ for the shifted graded module with $M(d)_n=M_{n+d}$ for all $n\in\Z$. When it will cause no confusion, we denote  $\Hom_{\textbf{Gr-A-mod}^{f}}(M,N)$ and $\Hom_{\textbf{Gr-mod-A}^{f}}(M,N)$ by $\hom_A(M,N)$. For $n\in\Z$, we define $\Hom^n_A(M,N)=\hom_A(M,N(-n))$. One can think of $\Hom_A^n(M,N)$ as the space of graded homomorphisms $f:M\rightarrow N$ which lower every homogeneous element's degree by $n$. The graded $\Hom$ functor is $\Hom^*_A(M,N)=\bigoplus_{n\in\Z} \Hom^n_A(M,N)$. The $i$-th right derived functors of $\Hom_A^*(M,-)$ and $\Hom_A^*(-,N)$ will be denoted respectively by $\Ext^i_A(M,-)$ and $\Ext^i_A(-,N)$. 

For each $i\in\N$, the $k$-vector space $\Ext_A^i(M,N)$ inherits a $\Z$-grading from the graded $\Hom$ functor. We call this the \emph{internal} grading. The homogeneous (internal) degree $j$ component of $\Ext_A^i(M,N)$ is denoted $\Ext^{i,j}_A(M,N)$. The vector space $\Ext_A(M,N)=\bigoplus_{i\in\N} \Ext_A^i(M,N)$ is therefore bigraded. 

If $L$, $M$, and $N$ are graded left $A$-modules, the Yoneda composition product gives a bilinear, associative pairing $$\star:\Ext_A^{i,j}(M,N)\otimes \Ext_A^{k,l}(L,M)\rightarrow \Ext_A^{i+k, j+l}(L,N)$$
If $L=M=N=k$, the Yoneda product gives the bigraded vector space $\Ext_A(k,k)$ the structure of a $k$-algebra. As mentioned in the introduction, we call this algebra the \emph{Yoneda algebra} of $A$, denoted $E(A)$. If $M=N=k$, then the Yoneda product makes $\Ext_A(L,k)$ into a graded left $E(A)$-module. Similarly, $\Ext_A(k,N)$ is a graded right $E(A)$-module.
The definition of the Yoneda algebra is unchanged if $L$, $M$, and $N$ are right modules. See Section \ref{factorK2} for more details.

\section{Koszul and $\K_2$ modules} 
\label{modules}

In this section we establish a few facts about extensions of Koszul and $\K_2$ modules. We also record useful characterizations of $\K_1$ and $\K_2$ modules in terms of minimal graded projective resolutions.

\begin{lemma}
\label{K2Extensions}
Let $0\rightarrow L\rightarrow M\rightarrow N\rightarrow 0$ be a graded exact sequence of graded left $A$-modules.
\begin{enumerate}
\item If  $L$ and $N$ are $\K_2$ $A$-modules and if the natural homomorphism ${\Ext_A^0(M,k)\rightarrow \Ext_A^0(L,k)}$ is surjective,  $M$ is a $\K_2$ $A$-module.
\item If  $A$ is a Koszul algebra and $L$ and $M$ are $\K_2$ $A$-modules and $N$ is generated in degrees strictly greater than the degrees in which $L$ is generated, then $N$ is a $\K_2$ $A$-module.
\end{enumerate}
\end{lemma}

\begin{proof}
Let $D(A)=D_2(A)$.
 If $N$ is generated in homogeneous degrees strictly greater than the degrees in which $L$ is generated, then the natural map ${g_0:\Ext_A^0(M,k)\rightarrow\Ext_A^0(L,k)}$ is surjective. Thus for both (1) and (2),  the long exact sequence in cohomology has the form
$$0\rightarrow \Ext_A^0(N,k)\xrightarrow{f_0} \Ext_A^0(M,k)\xrightarrow{g_0} \Ext_A^0(L,k)\xrightarrow{0} \Ext_A^1(N,k)\rightarrow \cdots$$
$$\xrightarrow{f_i}\Ext_A^i(M,k)\xrightarrow{g_i} \Ext_A^i(L,k)\xrightarrow{h_i} \Ext_A^{i+1}(N,k)\xrightarrow{f_{i+1}}\cdots$$
This is a complex of left $E(A)$-modules. Since $L$ is a $\K_2$ module, we have $\Ext_A^i(L,k)=D^i(A)\star\Ext_A^0(L,k)\rightarrow D^i(A)\star\Ext_A^1(N,k)= \Ext_A^{i+1}(N,k)$ is zero for all $i$. We therefore obtain a short exact sequence
$$0\rightarrow \Ext_A(N,k)\rightarrow \Ext_A(M,k)\rightarrow \Ext_A(L,k)\rightarrow 0$$
of left $E(A)$-modules. Statement (1) now follows.


To prove (2) we assume $A$ is Koszul and choose $p$ maximal such that $\Ext^{0,p}_A(L,k) \ne 0$. ÊBy hypothesis, $\Ext^0_A(N,k)$ is supported in internal degrees greater than $p$, from which it follows that for all $i\ge 0$, $\Ext^i_A(N,k)$ is supported in internal degrees greater than $p+i$. ÊChoose a bigraded vector space splitting $\Ext^0_A(M,k )= f_0(\Ext^0_A(N,k)) \oplus W$. The subspace Ê$ W$ is supported in internal degrees at most $p$. ÊSince $A$ is Koszul, $E^i(A)\star W$ is supported in internal degrees at most $p+i$. ÊHence if $f=\oplus_i f_i$, we have $f(\Ext_A(N,k)) \cap E(A)\star W = 0$. ÊSince $M$ is $\K_2$, ÊÊ$f(\Ext_A(N,k)) + E(A)\star W= \Ext_A(M,k)$. ÊThis splitting proves the result.

\end{proof}

The strictness in (2) cannot be weakened. See Example \ref{K2notCL}. The following Lemma should be well known.

\begin{lemma}
\label{KoszulExtensions}
Let $0\rightarrow L\rightarrow M\rightarrow N\rightarrow 0$ be a graded exact sequence of graded left $A$-modules.
\begin{enumerate}
\item If $L$ and $N$ are Koszul $A$-modules and the natural homomorphism ${\Ext_A^0(M,k)\rightarrow \Ext_A^0(L,k)}$ is surjective,  $M$ is a Koszul $A$-module.
\item If $L$ and $M$ are Koszul $A$-modules and the natural homomorphism ${\Ext_A^0(M,k)\rightarrow \Ext_A^0(L,k)}$ is surjective,  $N$ is a Koszul $A$-module.
\item If $M$ and $N$ are Koszul $A$-modules and the natural monomorphism $\Ext_A^0(N,k)\hookrightarrow \Ext_A^0(M,k)$ is an isomorphism, then $L$ is a Koszul $A$-module.
\end{enumerate}
\end{lemma}

\begin{proof}
Since Koszul modules are $\K_1$ modules, the proof of Lemma \ref{K2Extensions} shows that in (1) and (2), the long exact sequence in cohomology breaks into short exact sequences. Each conclusion follows immediately by considering degrees.

 For (3), suppose $N$ and $M$ are generated in homogeneous degree $d$ and consider the long exact sequence in cohomology
$$0\rightarrow \Ext_A^0(N,k)\xrightarrow{\isom} \Ext_A^0(M,k)\xrightarrow{0} \Ext_A^0(L,k)\rightarrow \Ext_A^1(N,k)\rightarrow \cdots$$
$$\rightarrow\Ext_A^i(M,k)\rightarrow \Ext_A^i(L,k)\rightarrow \Ext_A^{i+1}(N,k)\rightarrow\cdots$$
Since $\Ext_A^0(L,k)\hookrightarrow \Ext_A^1(N,k)$ and $N$ is a Koszul $A$-module, $\Ext_A^0(L,k)$ is concentrated in internal degree $d+1$. Since $M$ is a Koszul $A$-module, $\Ext_A^i(M,k)$ is concentrated in internal degree $d+i$, so  the map $\Ext_A^i(M,k)\rightarrow\Ext_A^i(L,k)$ is zero. Thus $\Ext_A^i(L,k)\hookrightarrow \Ext_A^{i+1}(N,k)$ for all $i$. Finally, since $\Ext_A^{i+1}(N,k)$ is concentrated in internal degree $d+i+1$, the same is true for $\Ext_A^i(L,k)$, hence $L$ is a Koszul $A$-module.
\end{proof}

We note that the surjectivity condition in (1) and (2) is satisfied for arbitrary graded modules when $N$ is generated in homogeneous degrees greater than or equal to the degrees in which $L$ is generated.
We do not expect an analog of Lemma \ref{KoszulExtensions}(3) for $\K_1$ or $\K_2$ modules. 
For example, submodules of $\K_1$ modules need not be $\K_1$, even with the hypothesis on $\Ext^0$. See Example \ref{submodule} below.

Next we recall the matrix criterion of  \cite{CS} for a left $A$-module $M$ to be a $\K_2$ module and state the analogous result for $\K_1$ modules. We need to introduce some notation. Let $V$ be a finite dimensional vector space, $A=T(V)/I$ a graded algebra and $M$ a graded left $A$-module. Let $Q^{\bullet}\rightarrow M$ be a minimal resolution of $M$ by graded projective left $A$-modules. Choose homogeneous bases for each $Q^i$ and let $M_i$ be the matrix of $d_i:Q^i\rightarrow Q^{i-1}$ with respect to these bases. Let $f_i$ be a lift of $M_i$ to a matrix over $T(V)_+$ with homogeneous entries. Let $L(f_i)$ denote $f_i$ mod $T(V)_{\ge 2}$. For $i\ge 0$, let $(f_{i+1}f_i)_{\text{ess}}$ denote the product $f_{i+1}f_i$ mod $I'$ where $I'=V\tensor I+I\tensor V$. Note that $f_0=0$. 

\begin{lemma}[\cite{CS}]
\label{K2matrix}
An $A$-module $M$ is a $\K_2$ $A$-module if and only if for all $0\le i< \text{pd}_A(M)$, the matrix $[(f_{i+1}f_i)_{\text{ess}}\ L(f_{i+1})]$ has linearly independent rows.
\end{lemma}

The following Proposition is proven by arguing as in Lemma 4.3 and Theorem 4.4 of \cite{CS}, replacing the trivial $A$-module with $M$.

\begin{prop}
\label{K2matrixCond}
Let $A$ be an $\N$-graded, connected algebra, and let $M$ be a graded $A$-module. For $0\le i\le pd_A(M)$ let $f_i$ be as above. Then $\Ext_A^i(M,k)=E^1(A)\star \Ext_A^{i-1}(M,k)$ if and only if $L(f_i)$ has linearly independent rows. Thus $M$ is a $\K_1$ $A$-module if and only if $L(f_i)$ has linearly independent rows for all $1\le i\le pd_A(M)$.
\end{prop}

 The following characterization of Koszul modules over quadratic algebras will be useful in Section \ref{CLmods}.

\begin{cor}
\label{K2ModsAreKoszul}
If $A$ is a quadratic algebra and $M$ is a graded left $A$-module, then $M$ is a $\K_2$ $A$-module if and only if $M$ is a $\K_1$ $A$-module. Furthermore, if $M$ is generated in a single homogeneous degree, then $M$ is a $\K_2$ $A$-module if and only if $M$ is a Koszul $A$-module.
\end{cor}

\begin{proof}
It suffices to prove that $M$ is a $\K_2$ $A$-module only if $M$ is a $\K_1$ $A$-module.
If $L(f_i)$ has linearly dependent rows for some $1\le i\le \text{pd}_A(M)$, then after changing basis we may assume the first row of $f_i$ contains no linear entries. This implies that the nonzero entries of the corresponding row of $f_i f_{i-1}$ have degree at least 3. As $A$ is quadratic, $I_{>2}= I'$. So the first row of $[ (f_{i}f_{i-1})_{\text{ess}}\ L(f_i)]$ is zero, and $M$ is not a $\K_2$ $A$-module by Proposition \ref{K2matrixCond}.

\end{proof}

\section{Componentwise Linear Resolutions}
\label{CLmods}

Let $M$ be a graded left $A$-module. Throughout this section, we additionally assume that $M$ is bounded below and let $b\in\Z$ denote the smallest integer such that $M_b\neq 0$. If the submodule $AM_i$ is a Koszul $A$-module for all $i$, we say that $M$ \emph{has a componentwise linear $A$-module resolution}. Our definition is motivated by the notion of componentwise linear ideal introduced in \cite{HH} and studied further in \cite{ReinerStamate}. In this section, we prove that all modules over Koszul algebras having a componentwise linear resolution are $\K_2$ (equivalently, $\K_1$) modules. We also characterize which $\K_2$ modules over Koszul algebras have componentwise linear resolutions. Following the notation of \cite{ReinerStamate}, let $M_{\la j\ra}=AM_j$ and for $i\le j$, let $M_{\la i,j \ra}=\sum_{t=i}^j AM_t$. 

\begin{lemma} 
\label{baseCase}
If $A$ is a Koszul algebra and $M_{\la j\ra}$ is a Koszul $A$-module for some $j\in\Z$, then
\begin{enumerate}
\item  $M_{\la j\ra}\cap M_{\la j+1\ra}$ is a Koszul $A$-module, and 
\item $M_{\la j+1\ra}$ is a Koszul $A$-module if and only if $M_{\la j,j+1\ra}/M_{\la j\ra}$ is a Koszul module.
\end{enumerate}
\end{lemma}

\begin{proof}
For $j\in\Z$ we have the exact sequence
$0\rightarrow M_{\la j\ra}\cap M_{\la j+1\ra}\rightarrow M_{\la j\ra}\rightarrow T\rightarrow 0$
where $T=M_{\la j\ra}/M_{\la j\ra}\cap M_{\la j+1\ra}$ is a trivial $A$-module. Hence $T$ is a Koszul $A$-module. Since $M_{\la j\ra}\cap M_{\la j+1 \ra}$ is concentrated in degrees $\ge j+1$, $\Ext_A^0(T,k)\isom \Ext_A^0(M_{\la j\ra}, k)$. Since $M_{\la j\ra}$ is a Koszul $A$-module, Lemma \ref{KoszulExtensions}(3) implies $M_{\la j\ra}\cap M_{\la j+1 \ra}$ is a Koszul $A$-module.

To prove (2), we consider the exact sequence
$$0\rightarrow M_{\la j\ra}\cap M_{\la j+1\ra}\rightarrow M_{\la j+1\ra}\rightarrow M_{\la j,j+1\ra}/M_{\la j\ra}\rightarrow 0$$ 
 Since all three modules are generated in homogeneous degree $j+1$, the  map $\Ext_A^0(M_{\la j+1\ra},k)\rightarrow \Ext_A^0( M_{\la j\ra}\cap M_{\la j+1\ra},k)$ is surjective. The result follows by Lemma \ref{KoszulExtensions}(1),(2).

\end{proof}

\begin{rmk}
\label{baseCase2}
In our applications of Lemma \ref{baseCase}, we require the slightly modified statements that $M_{\la d,j\ra}\cap M_{\la j+1\ra}$ is a Koszul $A$-module for all $d<j$ and $M_{\la j+1\ra}$ is a Koszul $A$-module if and only if $M_{\la d,j+1\ra}/M_{\la d, j\ra}$ is a Koszul $A$-module. These are obviously equivalent to (1) and (2) above.
\end{rmk}

\begin{lemma}
\label{intervalModulesK2}
If $A$ is a Koszul algebra and $M$ has a componentwise linear $A$-module resolution, then
 $M_{\la d,j\ra}$ is a $\K_2$ $A$-module for all $d,j\in\Z$ such that $d\le j$.
\end{lemma}

\begin{proof}
By Lemma \ref{baseCase}(2), $M_{\la d,j\ra}/M_{\la d,j-1\ra}$ is a Koszul $A$-module for all $j\in\Z$ and all $d<j$.
We consider the exact sequence
$$0\rightarrow M_{\la d,j-1\ra}\rightarrow M_{\la d,j\ra}\rightarrow M_{\la d,j\ra}/M_{\la d,j-1\ra}\rightarrow 0$$
and prove the result by induction on $j-d$. If $j=d$, $M_{\la d,j\ra}$ is a Koszul $A$-module by assumption. For the induction step, we assume $M_{\la d,j-1\ra}$ is  a $\K_2$ $A$-module. Since $M_{\la d,j\ra}/M_{\la d,j-1\ra}$ is a Koszul $A$-module, Lemma \ref{K2Extensions}(1) implies that $M_{\la d,j\ra}$ is a $\K_2$ $A$-module. The result follows by induction.

\end{proof}

\begin{prop} 
\label{CLimpliesK2}
Let $A$ be a Koszul algebra and let $M$ be an $A$-module having a componentwise linear resolution. If $b\in\Z$ is minimal such that $M_b\neq 0$, then $M_{\la b,j\ra}$ is a $\K_2$ $A$-module for all $j\ge b$. In particular, $M$ is a $\K_2$ $A$-module.
\end{prop}

\begin{proof}
The first statement follows immediately from Lemma \ref{intervalModulesK2}. To prove that $M$ is $\K_2$, we show  $\Ext_A^{p,b+q}(M,k)$ is $E(A)$-generated by $\Ext_A^{0,b+q-p}(M,k)$ by considering the long exact sequence in cohomology associated to the exact sequence $0\rightarrow M_{\la b,b+q-p\ra}\rightarrow M\rightarrow F\rightarrow 0$.
Since $F$ is generated in degrees $\ge b+q-p+1$, $\Ext_A^p(F,k)$ is concentrated in internal degrees $\ge b+q+1$. Thus the natural map
$\Ext_A^0(M,k)\rightarrow \Ext_A^0(M_{\la b,b+q-p\ra}, k)$ is surjective in internal degrees $\le b+q-p$ and
$\Ext_A^{p,b+q}(M,k)\hookrightarrow \Ext_A^{p,b+q}(M_{\la b,b+q-p\ra},k)$. Since the morphisms in the long exact sequence respect the left $E(A)$-module structure, the result follows from the assumption that $M_{\la b, b+q-p\ra}$ is a $\K_2$ $A$-module.

\end{proof}

Let $M$ be a bounded below graded left $A$-module and recall we let $b$ be the smallest integer such that $M_b\neq 0$. We say $M$ is  a \emph{strongly $\K_2$} $A$-module if $M_{\la b,j\ra}$ is a $\K_2$ $A$-module for all $j\ge b$. We note that if $M$ is strongly $\K_2$ and $A$ is a quadratic algebra, then $M_{\la b\ra}$ is a Koszul $A$-module by Corollary \ref{K2ModsAreKoszul}. For an example of a $\K_2$ module over a Koszul algebra that is not strongly $\K_2$, see Example \ref{K2notCL}.

\begin{lemma}
\label{stronglyK2FactorMod}
Let $A$ be a Koszul algebra and $M$ a strongly $\K_2$ $A$-module. Let $b$ be the smallest integer such that $M_b\neq 0$. Then $F=M_{\la b,j\ra}/M_{\la b,j-1\ra}$ is a Koszul $A$-module for all $j>b$.
\end{lemma}

\begin{proof} Let $j>b$.
Applying Lemma \ref{K2Extensions}(2) to the exact sequence
$$0\rightarrow M_{\la b,j-1\ra}\rightarrow M_{\la b,j\ra}\rightarrow F\rightarrow 0$$ 
where $F=M_{\la b,j\ra}/M_{\la b,j-1\ra}$ shows that $F$ is a $\K_2$ $A$-module.  Since $F$ is generated in homogeneous degree $j$, Corollary \ref{K2ModsAreKoszul} implies that $F$ is a Koszul $A$-module.

\end{proof}

We now prove the converse of Proposition \ref{CLimpliesK2}.

\begin{prop}
\label{converse}
If $A$ is a Koszul algebra and if $M$ is a strongly $\K_2$ $A$-module, then $M$ has a componentwise linear resolution.
\end{prop}

\begin{proof}
Let $b$ be the smallest integer such that $M_b\neq 0$.
Since $A$ is quadratic, $M_{\la b\ra}$ is a Koszul $A$-module by Corollary \ref{K2ModsAreKoszul}. Assume inductively that $M_{\la b+t\ra}$ is a Koszul $A$-module for $t\ge 0$ and 
consider the exact sequence
$$0\rightarrow M_{\la b,b+t\ra}\cap M_{\la b+t+1\ra}\rightarrow M_{\la b+t+1\ra}\rightarrow M_{\la b,b+t+1\ra}/M_{\la b,b+t\ra}\rightarrow 0$$
 By the induction hypothesis and by Lemma \ref{baseCase}(1), ${M_{\la b,b+t\ra}\cap M_{\la b+t+1\ra}}$ is a Koszul $A$-module. Lemma \ref{stronglyK2FactorMod} implies that $M_{\la b,b+t+1\ra}/M_{\la b,b+t\ra}$ is a Koszul $A$-module. As all three modules are generated in homogeneous degree $b+t+1$, the result follows by Lemma \ref{KoszulExtensions}(1).

\end{proof}

Combining Propositions \ref{CLimpliesK2} and \ref{converse}, we obtain a characterization of modules over Koszul algebras having componentwise linear resolutions.

\begin{cor}
If $A$ is a Koszul algebra, then $M$ has a componentwise linear $A$-module resolution if and only if $M$ is a strongly $\K_2$ $A$-module.
\end{cor}

\bigskip

\section{$\K_2$ Factor Algebras}
\label{factorK2}

Recall that the \emph{socle} of an $A$-module $M$, denoted $\soc(M)$ is the unique maximal semisimple submodule of $M$. If $A$ is a graded $k$-algebra, a simple module in $\textbf{Gr-A-mod}^{f}$ is isomorphic to ${k}(d)$ for some $d\in\Z$. Thus  we have $\soc(M)=M^{A_+}=\{m\in M: am=0, \forall a\in A_+\}$. More generally, we define $M^I=\{m\in M: am=0, \forall a\in I\}$ for any ideal $I$ in $A$.

Let $k$ be a field and let $V$ be a finite dimensional $k$-vector space on basis $X=\{x_1,\ldots,x_n\}$. Let $A=T(V)/R$ be a graded $k$-algebra generated by $V$. We identify $x_i$ with its image in $A$. We give the algebra $A$ the usual $\N$-grading by tensor degree with $\deg(x_i)=1$ for $i=1,\ldots,n$. Let $\I\subset A_{\ge 2}$ be a graded ideal generated by homogeneous elements of $A$ and let $B=A/ \I$. The algebra $B$ inherits an $\N$-grading from $A$. 


Let $(P^{\bullet},\partial_P)$ be a minimal resolution of the trivial right $A$-modue $k_A$ by 
graded projective right $A$-modules with degree 0 differential. The augmentation map is denoted $\epsilon:P^0\rightarrow k$. We denote the graded dual $\Hom$ functor $\Hom_k(-,k)$ by $-^*$.

\begin{lemma}
\label{injRes}
The complex $I_{\bullet}=(P^{\bullet})^*$ is a resolution of $_Ak\isom (k_A)^*$ by graded injective left $A$-modules. The coaugmentation $\epsilon^{\vee}:k\rightarrow I_0$ is given by $1\mapsto \epsilon$.
\end{lemma}

\begin{proof}
Since $k$ is semisimple, the functor $-^*$ is exact, hence $I_{\bullet}$ is an $A$-linear resolution of $k^*$ with the stated coaugmentation. 
That $I_j$ is injective for each $j$ follows easily from the exactness of $-^*$ and the fact that $P^j$ is locally finite-dimensional.

\end{proof}

We denote the induced differential on $I_{\bullet}$ by $\partial^I$. We assume that $P^0=A$, that $P^1=A(-1)^{\oplus n}$, and that homogeneous $A$-bases are chosen for $P^1$ and $P^0$ such that the matrix of the differential $P^1\rightarrow P^0$ is given by left multiplication by $(x_1\ \cdots\ x_n)$.  

We take $H^*(\Hom_A(k,I_{\bullet}))$ as our model for the bigraded Yoneda Ext-algebra $E(A)=\bigoplus E^{n,m}(A)=\bigoplus\Ext_A^{n,m}(k,k)$.  We note that $\soc(I_j)$ is generated by the graded $k$-linear duals of an $A$-basis for $P^j$ and thus $\soc(I_j)\subset \im\ \partial_j^I$. Identifying $\Hom_A(k,I_j)$ with $\soc(I_j)$, we conclude that the differential on $\Hom_A(k,I_{\bullet})$ is zero and $E^j(A)=\Hom_A(k,I_j)\isom \soc(I_j)$.

We let $J_{\bullet}=\Hom_A(B,I_{\bullet})$ and we take $H^*(J_{\bullet})$ as our model for $\Ext_A(B,k)$. We denote  by $\partial^J$ the differential on $J_{\bullet}$ induced by $\partial^I$. Clearly, $g\in\soc(J_p)$ if and only if $g(1)\in\soc(I_p)$.

Let $(Q^{\bullet},{\partial_Q})$ be a minimal resolution of the trivial $B$-module $_Bk$ by graded projective left $B$-modules. We take $H^*(\Hom_B(Q^{\bullet},k))=\Hom_B(Q^{\bullet},k)$ as our model for $E(B)$.

The Cartan-Eilenberg change-of-rings spectral sequence
$$E_2^{p,q}=\Ext_B^p(k,\Ext_A^q(B,k))\Rightarrow \Ext_A^{p+q}(k,k)$$
 is the spectral sequence associated with the first-quadrant double cocomplex
 $$A^{p,q}=\Hom_B(Q^p,\Hom_A(B,I_q)).$$
 This is a spectral sequence of bigraded $E(A)-E(B)$ bimodules (see Lemmas 6.1 and 6.2 of \cite{CS}).  The horizontal differential $d_h$ is precomposition with $\partial_Q$. The vertical differential $d_v$ is  composition with $\partial^J$. The left $E(A)$ action on $E_2^{p,q}$ is induced by the Yoneda product $\Ext_A^m(k,k)\otimes\Ext_A^n(B,k)\rightarrow \Ext_A^{m+n}(B,k)$. The right $E(B)$ action is given by the Yoneda product $$\Ext_B^m(k,\Ext_A(B,k))\otimes \Ext_B^n(k,k)\rightarrow \Ext_B^{m+n}(k,\Ext_A(B,k))$$
 We will make these module actions more explicit below.
 
 We recall the standard ``staircase argument'' for double cocomplexes. Let $(A^{p,q},d_h,d_v)$ be a first-quadrant double cocomplex. For $N>1$ and for $p,q\ge 0$ let
 $\cS_N^{p,q}$ denote the subset of $\Pi_{i=0}^{N-1} A^{p+i,q-i}$ consisting of $N$-tuples $(a_1,\ldots, a_N)$ such that $d_va_1=0$ and  $d_ha_i+d_va_{i+1}=0$ for all  $1\le i\le N-1$.
 
 \begin{lemma}
 \label{staircase}
 If $(a_1,\ldots, a_N)\in \cS_N^{p,q}$, then $a_1$ survives to $E_N^{p,q}$ and $d_N([a_1])=[d_ha_N]$.
 Every class in $E_N^{p,q}$ can be represented by an element of $\cS_N^{p,q}$, and the representation is unique modulo elements in $\cS_N^{p,q}$ of the form
 \begin{itemize}
\item $(d_h \alpha,0,\ldots,0)$ where $\alpha\in A^{p-1,q}$, 
\item  $(0,\ldots,0,d_v \beta, d_h\beta, 0,\ldots 0)$ where  $\beta\in A^{p+i-1,q-i}$ and $d_v\beta$ is in position $i$ for $1\le i\le n-1$, and
\item $(0,\ldots,0,\delta)$ where $\delta\in A^{p+N-1,q-N+1}$.
\end{itemize}
 \end{lemma}

Applying Lemma \ref{staircase} to  $A^{p,q}=\Hom_B(Q^p,\Hom_A(B,I_q))$, we can represent equivalence classes of elements in $E_N^{p,q}$ by diagrams
 \begin{diagram}
Q^{p+N-1}&\rTo^{\partial_Q^{p+N-1}}&  \cdots &\rTo& Q^{p+2}  &\rTo^{\partial_Q^{p+2}}&Q^{p+1} & \rTo^{\partial_Q^{p+1}} & Q^p \\
\dTo_{a_N}&&\cdots& &\dTo& &\dTo_{a_2} &        & \dTo_{a_1} \\
J_{q-N+1}&\rTo^{\partial^J_{q-N+2}} &\cdots &\rTo & J_{q-2} &\rTo^{\partial^J_{q-1}}&J_{q-1} & \rTo^{\partial^J_{q}} & J_{q}\\
\end{diagram}
 where the squares anticommute and $\partial^J_{q+1}a_1=0$. If $(a_1,a_2,\ldots,a_N)$ and $(b_1,b_2,\ldots,b_N)$ represent the same class in $E_N^{p,q}$, we write $(a_1,a_2,\ldots,a_N)\sim (b_1,b_2,\ldots,b_N)$.

As in Section 6 of \cite{CS}, we can describe the $E(A)-E(B)$ bimodule structure on the 
spectral sequence at the staircase level by composing diagrams. A class $[\z]\in E^k(B)$ is represented by a $B$-linear homomorphism $\z:Q^k\rightarrow k$. Lifting this map through the complex $Q^{\bullet}$ by projectivity, we obtain a commutative diagram of the form
\begin{diagram}
Q^{k+p+N-1}& \rTo&\cdots &\rTo&Q^{k+p}  &\cdots  &\rTo&Q^{k} & & \\
\dTo_{\z_{p+N-1}}&&&&\dTo_{\z_p}&&& \dTo_{\z_0} & \rdTo^{\z}&\\
Q^{p+N-1}&\rTo &\cdots &\rTo &Q^{p} &\cdots  &\rTo&Q^0 & \rTo &k\\
\end{diagram}
Pre-composing with a representative of a class in $E_N^{p,q}$, we obtain
\begin{diagram}
Q^{k+p+N-1}& \rTo&\cdots &\rTo&Q^{k+p+1}&\rTo&Q^{k+p}  \\
\dTo_{\z_{p+N-1}}&&&&\dTo_{\z_{p+1}}&&\dTo_{\z_p}\\
Q^{p+N-1}&\rTo &\cdots&\rTo&Q^{p+1} &\rTo &Q^{p} \\
\dTo_{a_N}&&&&\dTo_{a_2}& &\dTo_{a_1} \\
J_{q-N+1}&\rTo&\cdots&\rTo &J_{q-1}&\rTo & J_{q}  \\
\end{diagram}
If we define $[(a_1,\ldots, a_N)]\star[\z] = [(a_1\z_p,\ldots,a_N\z_{p+N-1})]$ for all $p, q$, and $k\ge 0$, then $\star$ gives $E_N^{*,q}$ a well-defined right $E(B)$-module structure.
For fixed $p, q, k$, we denote the set of all such products by $E_N^{p,q}\star E^k(B)$.

Similarly, a class in $E^m(A)$ is represented by an $A$-linear homomorphism $\g:k\rightarrow I_m$ which can be lifted through the complex $I_{\bullet}$ by injectivity to obtain a commutative diagram
\begin{diagram}
k&\rTo&I^{0}& \rTo&\cdots &\rTo&I_{q-N+1}  &\cdots  &\rTo&I_q  \\
&\rdTo_{\g}&\dTo_{\g_{0}}&&&&\dTo_{\g_{q-N+1}}&&& \dTo_{\g_q} \\
&&I_m&\rTo &\cdots &\rTo &I_{m+q-N+1} &\cdots  &\rTo&I_{m+q} \\
\end{diagram}

Applying the functor $\Hom_A(B,-)$ to the last $N$ terms in this diagram, we obtain the commutative diagram
\begin{diagram}
J_{q-N+1}&\rTo &\cdots&\rTo&J_{q-1} &\rTo &J_{q} \\
\dTo_{\widetilde{\g}_{q-N+1}}&&&&\dTo_{\widetilde{\g}_{q-1}}& &\dTo_{\widetilde{\g}_q} \\
J_{m+q-N+1}&\rTo&\cdots&\rTo &J_{m+q-1}&\rTo & J_{m+q}  \\
\end{diagram}
Post-composing with a representative of $E_N^{p,q}$, we get
\begin{diagram}
Q^{p+N-1}&\rTo &\cdots&\rTo&Q^{p+1} &\rTo &Q^{p} \\
\dTo_{a_N}&&&&\dTo_{a_2}& &\dTo_{a_1} \\
J_{q-N+1}&\rTo&\cdots&\rTo &J_{q-1}&\rTo & J_{q}  \\
\dTo_{\widetilde{\g}_{q-N+1}}&&&&\dTo_{\widetilde{\g}_{q-1}}& &\dTo_{\widetilde{\g}_q} \\
J_{m+q-N+1}&\rTo&\cdots&\rTo &J_{m+q-1}&\rTo & J_{m+q}  \\
\end{diagram}
If we define $[\g]\star[(a_1,\ldots, a_N)]=[(\widetilde{\g}_qa_1,\ldots,\widetilde{\g}_{q-N+1}a_N)]$ for all $p, q$ and $m\ge 0$, then $\star$ gives $E_N^{p,*}$ a well-defined left $E(A)$-module structure. For fixed $p, q, m$, we denote the set of all such products by $E^m(A)\star E_N^{p,q}$. As $E(A)$ always acts on the left of $E_N$ and $E(B)$ always acts on the right, the meaning of the symbol $\star$ should always be clear from context.

 \begin{lemma}
 \label{triviality}
 If 
 the natural left  $B$-module structure 
  on $\Ext_A(B,k)$ is trivial, there results an isomorphism of bigraded bimodules
 $$\bigoplus_{p,q}\ E_2^{p,q}\ \isom\ \bigoplus_{p,q}\ \Ext_A^q(B,k)\tensor E^p(B)$$
 \end{lemma}
 
 \begin{proof} 
 Because we work with locally finite-dimensional modules, for every $p$, $q\ge 0$ we have an isomorphism
 $$\Ext_A^q(B,k)\otimes \Hom_k(Q^p,k)\rightarrow \Hom_k(Q^p,\Ext_A^q(B,k))$$ given by $[\z]\otimes f\mapsto f_{[\z]}$ where $f_{[\z]}(q)=f(q)[\z]$. Since $\Ext_A(B,k)$ is a trivial $B$-module, this isomorphism restricts to an isomorphism
  $$\Lambda^{pq}:\Ext_A^q(B,k)\otimes \Hom_B(Q^p,k)\rightarrow \Hom_B(Q^p,\Ext_A^q(B,k))$$ 
 Indeed, if $b\in B_+$, $$f_{[\z]}(bq) = f(bq)[\z]=(bf(q))[\z]=0[\z]=0=b(f_{[\z]}(q))$$ so the restricted map is well-defined. Let $\widetilde{\partial_Q}$ denote the differential induced on the graded vector space $\Hom_B(Q^{\bullet},\Ext_A^q(B,k))$ by $\partial_Q$. By the minimality of $Q^{\bullet}$, $\im\ \partial_Q^{p+1}\subset B_+Q^p$ so $\widetilde{\partial_Q}(f_{[\z]})=f_{[\z]}\partial_Q=0$. Thus we have  $\Ext_B(Q^p,\Ext_A^q(B,k))=\Hom_B(Q^p,\Ext_A^q(B,k))$  and $\Lambda=\bigoplus \Lambda^{pq}$ yields the desired isomorphism.
 
 \end{proof}
 
 
 \begin{rmk}
 \label{remark}
 It is important to note that the induced $E(A)-E(B)$ bimodule structure on $\Ext_A(B,k)\otimes E(B)$ that results from the isomorphism of Lemma \ref{triviality} is precisely the structure determined by the usual Yoneda product on each tensor component. Our use of the notation $\star$ for the spectral sequence module structure is therefore consistent with its prior use to indicate the Yoneda products on $\Ext_A(B,k)$ and $E(B)$. 
 \end{rmk}

We also note that the hypothesis that $\Ext_A(B,k)$ is a trivial left $B$-module is satisfied in many important cases, including when $A$ is commutative or graded-commutative. 
%

 
 We prove two lemmas relating to representatives in $\cS_N^{p,q}$. 
 
 \begin{lemma}
 \label{lifting}
Let $[\gamma]\in E^1(A)$ and let  $\widetilde{\gamma}_0:J_0\rightarrow J_1$ be induced by $\gamma$. 
 Let $\phi:Q^p\rightarrow J_0$ be a homomorphism of $B$-modules. If the composition $Q^p\xrightarrow{\phi} J_0\xrightarrow{\partial^J_1} J_1$ is zero, then $\phi$ factors through $_Bk$ and $\im\ \widetilde{\g}_0\phi\subset \soc(J_1).$ 
 
 
 \end{lemma}
 
 \begin{proof}
 If $\phi\partial_1^J=0$, then $\im\ \phi \subset \ker\ \partial_1^J=\im\ \widetilde{\e^{\vee}}$ where $\widetilde{\e^{\vee}}:{_Bk}\rightarrow J_0$ is induced by $\e^{\vee}$. Thus $\phi$ factors through $_Bk$ and $\im\ \phi\subset \soc(J_0)$. Since $\im\ \g_0\e^{\vee}=\im\ \g\subset \soc(I_1)$,  we have $\im\ \widetilde{\g_0}\phi\subset \soc(J_1)$.
 

 \end{proof}
 
 Recall that we fixed an $A$-basis for $P^1$ such that $P^1\rightarrow P^0$ is given by left multiplication by $(x_1\ x_2\ \cdots\ x_n)$. Denote the elements of this $A$-basis by $\varepsilon_1,\ldots,\varepsilon_n$. Denote the $k$-linear graded duals of these basis elements by $\varepsilon_1^*,\ldots,\varepsilon_n^*$. Then $\soc(I_1)$ is the trivial $A$-module generated by $\{  \varepsilon_1^*,\ldots,\varepsilon_n^*\}$.

 \begin{lemma}
 \label{lifting2}
 Let $[\g]\in E^1(A)$ and let $(f,g)\in\cS_2^{p,1}$. Let $\widetilde{\g_0}:J_0\rightarrow J_1$ and $\widetilde{\g_1}:J_1\rightarrow J_2$ be induced by $\g$. Then there exists $h\in A^{p+2,0}$ such that $(\widetilde{\g}_1f,\widetilde{\g}_0g, h)\in \cS_3^{p,2}$. Furthermore, $h$ can be chosen such that, for any vector $e$ in a $B$-basis for $Q^{p+2}$, $h(e)(1)\in I_0^{A_{\ge 2}}$.
 \end{lemma}
 
 \begin{proof}
 Let $\phi=d_hg=g\partial_Q^{p+2}$. 
 We have $\partial_1^J\phi=d_vd_hg=-d_hd_vg=d_h^2f=0$. 
 It suffices to define $h$ on an arbitrary element $e$ of a $B$-basis for $Q^{p+2}$.
 By Lemma \ref{lifting}, $\widetilde{\g_0}\phi(e)\in\soc(J_1)$ so $\widetilde{\g_0}\phi(e)(1)\in\soc(I_1)\subset \im\ \partial_1^I$. Let $u\in I_0$ such that $\partial_1^I(u)=\widetilde{\g_0}\phi(e)(1)$. Since $\partial_1^I(u)=u\partial_P^1$ is annihilated by $A_+$ and since $\partial_P^1$ is multiplication by $(x_1, \ldots, x_n)$, we can choose $u\in I_0^{A_{\ge 2}}$. We define an $A$-module homomorphism $h(e):{_AB}\rightarrow I_0$ by $h(e)(1)=u$. This is well-defined since $\I\subset A_{\ge 2}$. Finally, we have
 $$(d_vh)(e)=\partial_1^J(h(e))=\partial_1^Ih(e)=\widetilde{\g_0}\phi(e)=\widetilde{\g_0}d_hg(e)=d_h(\widetilde{\g_0}g)(e)$$
so $(\widetilde{\g}_1f,\widetilde{\g}_0g, h)\in \cS_3^{p,2}$.
 
 
 \end{proof}





\begin{lemma}
\label{E2page}
Let $A$ be a graded $k$-algebra. Let $\I$ be a graded ideal of $A$ and let $B=A/\I$. Assume $B$ acts trivially on $\Ext_A(B,k)$.
If $\Ext_A(\I,k)$ is generated as a left $E(A)$-module by $\Ext_A^0(\I,k)$, then for any $p$ and for $q\ge 2$, the spectral sequence differential $d_2^{p,q}$ is zero.  
 \end{lemma}
 
 \begin{proof}
If $\Ext_A(\I,k)$ is $E(A)$-generated by $\Ext_A^0(\I,k)$, $\bigoplus_{q>0}\Ext_A^q(B,k)$ is generated as a left $E(A)$-module by $\Ext_A^1(B,k)$.
By Lemma \ref{triviality}, we have $E_2^{p,q}\cong \Ext_A^q(B,k)\tensor E^p(B)$. By Remark \ref{remark}, we have  $E_2^{p,q}=E^{q-1}(A)\star E_2^{0,1}\star E^p(B)$ for $q\ge 1$. Thus any element of $E_2^{p,q}$ can be represented as a sum of diagrams of the form 
\newpage

\begin{diagram}
Q^{p+1} & \rTo^{\partial_Q^{p+1}} & Q^p \\
\dTo      &          & \dTo \\
Q^1 & \rTo^{\partial_Q^1} & Q^0 \\
\dTo_{a_2} &        & \dTo_{a_1} \\
J_{0} & \rTo^{\partial^J_1} & J_{1}\\
\dTo &            & \dTo \\
J_{q-1} & \rTo^{\partial^J_{q}} & J_q\\
\end{diagram}
which represent Yoneda products in $E^{q-1}(A)\star E_2^{0,1}\star E^p(B)$.
 
Since the spectral sequence differential respects the bimodule structure, it suffices to show $E^{q-1}(A)\star \im(d_2^{0,1})=0$.
Given a pair $(a_1,a_2)\in \cS_2^{0,1}$ representing a class in $E_2^{0,1}$, we have $d_2^{0,1}[a_1]=[d_ha_2]$. Let $\phi=a_2\partial_Q^2$ and $[\gamma]\in E^{q-1}(A)$. 
Since $d_vd_h(\widetilde{\g}_0a_2)=-d_h(\partial_q^J\widetilde{\g}_0a_2)=-d_h(\widetilde{\g}_1d_va_2)=d_h(\widetilde{\g}_1d_ha_1)=\widetilde{\g}_1a_1\partial_Q^1\partial_Q^2=0$,  we have $(\widetilde{\g}_0\phi,0)\in\cS_2^{2,q-1}$.
But $\widetilde{\g}_0\phi=d_h(\widetilde{\g}_0a_2)$ and $\widetilde{\g}_0a_2\in A^{1,q-1}$.  By Lemma \ref{staircase}, $(\widetilde{\g}_0\phi,0)$ represents $0$ in $E_2^{2,q-1}$. Hence $[\gamma]\star[\phi]=0$ for all $[\gamma]\in E^{q-1}(A)$ and all $[\phi]\in \im\ d_2^{0,1}$, so $d_2^{p,q}=0$.

\end{proof}

Though the differential $d_2^{p,q}$ vanishes for $q\ge 2$, the differential $d_2^{0,1}$ is likely nonzero. Thus the decomposition $E_2^{p,q}=E^{q-1}(A)\star E_2^{0,1}\star E^p(B)$ does not immediately translate to the $E_3$ page. However, if $_A\I$ is a $\K_1$ $A$-module, we obtain a similar result. Recall that $D_1(A)$ is the subalgebra of $E(A)$ generated by $E_1(A)$.

\begin{lemma}
\label{E2=E3} 
Let $A$ be a graded $k$-algebra and  $\I$ a graded ideal such that $_A\I$ is a $\K_1$ $A$-module and $B=A/\I$ acts trivially on $\Ext_A(B,k)$. Then for any $p$ and for $q\ge 2$, we have $E_2^{p,q}= E_3^{p,q}$ and $E_3^{p,q}=D_1^{q-2}(A)\star E_3^{0,2}\star E^p(B)$.
\end{lemma}

\begin{proof}
By Lemma \ref{E2page}, $d_2^{p,q}$ and $d_2^{p-2,q+1}$ are both zero for $q\ge 2$, so $E_2^{p,q}=E_3^{p,q}$ for any $p$ and for $q\ge 2$. Since $_A\I$ is a $\K_1$ module, $\Ext_A^q(\I,k)=E^1(A)\star\Ext_A^{q-1}(\I,k)$ for all $q>0$. Thus by induction, 
$$E_3^{p,q}=E_2^{p,q}=D_1^{q-2}(A)\star E_2^{0,2}\star E^p(B)=D_1^{q-2}(A)\star E_3^{0,2}\star E^p(B).$$
\end{proof}

The following technical lemma and its analog for the $E_N$ page (see Lemma \ref{ENpage}) show that if $_A\I$ is a $\K_1$ $A$-module, the staircase representations of Lemma \ref{staircase} have a particularly nice form.

\begin{lemma}
\label{isom}
Let $A$ and $B$ be graded $k$-algebras as in Lemma \ref{E2=E3}.
Then every class in $E_3^{p,q}$ can be represented as a sum of diagrams of the form
\begin{diagram}
Q^{p+2} &\rTo^{\partial_Q^{p+2}}&Q^{p+1} & \rTo^{\partial_Q^{p+1}} & Q^p \\
\dTo_{\z_{p+2}}& &\dTo_{\z_{p+1}} &        & \dTo_{\z_p} \\
Q^2 &\rTo^{\partial_Q^2}&Q^1 & \rTo^{\partial_Q^1} & Q^0 \\
\dTo(0,2)^{h}& &\dTo_{g} &        & \dTo_{f} \\
&&J_{0} & \rTo^{\partial^J_1} & J_{1}\\
&&\dTo_{\widetilde{\g_0}} &            & \dTo_{\widetilde{\g}_1} \\
J_{0}&\rTo^{\partial^J_1}&J_{1} & \rTo^{\partial^J_2} & J_2\\
\dTo_{\widetilde{\g'}_0}&&\dTo_{\widetilde{\g'}_1} &            & \dTo_{\widetilde{\g'}_2} \\
J_{q-2}&\rTo^{\partial^J_{q-1}}&J_{q-1} & \rTo^{\partial^J_q} & J_q\\
\end{diagram}
where the maps $\z_i$ are induced by some $[\z]\in E^p(B)$ and the maps $\widetilde{\g'}_k$ are induced by some $[\g']\in E^{q-2}(A)$.
 Furthermore, the map $h$ may be chosen so that, for any basis vector $e\in Q^2$, $h_i(e)(1)\in I_0^{A_{\ge 2}}$.
\end{lemma}


\begin{proof}
In light of Lemma \ref{E2=E3}, it suffices to prove the statement for $E_3^{0,2}$.
If $(a_1,a_2,a_3)$ represents a class in $E_3^{0,2}$, then $(a_1,a_2)$ represents a class in $E_2^{0,2}$. Since $E_2^{0,2}=E^1(A)\star E_2^{0,1}$, the class $[a_1]\in E_2^{0,2}$ can be represented as a sum of diagrams of the form
\begin{diagram}
Q^1 & \rTo^{\partial_Q^1} & Q^0 \\
\dTo_{g} &        & \dTo_{f} \\
J_{0} & \rTo^{\partial^J_1} & J_{1}\\
\dTo_{\widetilde{\g}_0} && \dTo_{\widetilde{\g}_1} \\
J_{1} & \rTo^{\partial^J_2} & J_2\\
\end{diagram}
where $(f,g)$ represents a class in $E_2^{0,1}$ and $\g$ represents a class in $E^1(A)$. 
%
Let $\alpha$, $\beta$, and $\delta$ be as in Lemma \ref{staircase} such that
$(a_1,a_2)=\sum_i (\widetilde{(\gamma_i)}_1 f_i, \widetilde{(\gamma_i)}_0 g_i)+(d_h\alpha,0)+(d_v\beta, d_h\beta)+(0,\delta)$. 

By Lemma \ref{lifting2}, there exist maps $h_i$ such that 
 $d_h ( \widetilde{(\gamma_i)}_0g_i)+d_vh_i=0$. Since $d_ha_2+d_va_3=0$, we have 
\begin{align*}
-d_va_3=d_ha_2&=d_h\left(\sum_i \widetilde{(\gamma_i)}_0g_i+d_h\beta+\delta\right)\\
&=-\sum_i d_vh_i+d_h\delta
\end{align*}
so $d_v(a_3-\sum_i h_i) + d_h\delta=0$. Since $d_v\delta =0$, Lemma \ref{staircase} implies that $(0,\delta, a_3-\sum_i h_i)$ represents a class in $E_3^{0,2}$. Evidently, it represents 0, as do $(d_h\a, 0, 0)$ and $(d_v\beta,d_h\beta,0)$. Thus we have
\begin{align*}
\sum_i ( \widetilde{(\gamma_i)}_1f_i, \widetilde{(\gamma_i)}_0g_i,h_i)&+(d_h\alpha,0,0)+(d_v\beta,d_h\beta,0)+\left(0,\delta,a_3-\sum_i h_i\right)\\
&=\left(a_1,a_2-\delta, \sum_i h_i\right)+\left(0,\delta,a_3-\sum_i h_i\right)\\
&=(a_1,a_2,a_3)
\end{align*}
and $(a_1,a_2,a_3)\sim \sum_i ( \widetilde{(\gamma_i)}_1f_i, \widetilde{(\gamma_i)}_0g_i,h_i)$ as desired. By Lemma \ref{lifting2}, we may assume the $h_i$ have the property that $h_i(e)(1)\in I_0^{A_{\ge 2}}$ for each basis vector $e\in Q^2$.

\end{proof}

We now state a version of Lemmas \ref{E2page}-\ref{isom} for the $E_N$ page. The proof, which we omit, is by induction using the same arguments as in those lemmas, which generalize in the obvious way  if we assume that $_A\I$ is a $\K_1$ module.
 \newpage
 
 \begin{lemma}
 \label{ENpage}
 Let $A$ be a graded algebra. Let $\I$ be a graded ideal of $A$. Assume $_A\I$ is a $\K_1$ $A$-module and $B=A/\I$ acts trivially on $\Ext_A(B,k)$.
Let $N\ge 2$. For any $p\ge 0$ and for $q\ge N-1$, any element of $E_N^{p,q}$ can be represented by a sum of diagrams of the form

 \begin{diagram}
Q^{p+N} &\rTo &\cdots&\rTo &Q^{p+2}&\rTo&Q^{p+1}&\rTo&Q^p\\
\dTo_{\z_{p+N}}        &         &           &      &\dTo_{\z_{p+2}}         &        &\dTo_{\z_{p+1}}      &       &\dTo_{\z_p}\\
Q^{N-1}  &\rTo & \cdots &\rTo&Q^2         &\rTo  &Q^1       &\rTo&Q^0\\
\dTo(0,3)_{b_N}&      &&   &\dTo(0,2)_{b_3}&         & \dTo_{b_2}     &       &\dTo_{b_1}\\
             &&&          &                &          &J_0       &\rTo&J_1\\
           &     &&    &               &           &\dTo_{\widetilde{(\g_1)}_0}     &        &\dTo_{\widetilde{(\g_1)}_1}\\
          &          &&&J_0         &\rTo   &J_1      &\rTo  &J_2\\
         &       &&  &\dTo_{\widetilde{(\g_2)}_0}       &           &\dTo_{\widetilde{(\g_2)}_1}     &        &\dTo_{\widetilde{(\g_2)}_2}\\
         &      &           &         &\vdots        &          &\vdots    &         &\vdots\\
        &      &           &         &\dTo        &          &\dTo    &         &\dTo\\
J_0        &\rTo &\cdots&\rTo  &J_{N-3}         &\rTo   &J_{N-2}     &\rTo  &J_{N-1}\\
\dTo_{\widetilde{\g}_{0}}        & &&  &\dTo_{\widetilde{\g}_{N-3}}         &   &\dTo_{\widetilde{\g}_{N-2}}     &  &\dTo_{\widetilde{\g}_{N-1}}\\
J_{q-N+1}        &\rTo &\cdots&\rTo  &J_{q-2}         &\rTo   &J_{q-1}     &\rTo  &J_q\\
\end{diagram}
where the maps $\z_i$ are defined by the action of some $[\z]\in E^p(B)$ on $E_N^{0,q}$, the maps $\widetilde{\g}_k$ are defined by the action of some $[\g]\in E^{q-N+1}(A)$ on $E_N^{0,N-1}$,  and for each $1\le j\le N-2$ the maps $\widetilde{(\g_j)}_k$ are defined by the action of some $[\g_j]\in E^1(A)$ on $E_{j+1}^{0,j}$. Furthermore, for $j>2$, the $b_j$ can be chosen such that, for any vector $e$ in a basis for $Q^{j-1}$, $b_j(e)(1)$ is a linear form. The differential $d_N^{p,q}=0$ for all $q\ge N$. We have $E_N^{p,q}= E_{N+1}^{p,q}$ and  $E_{N+1}^{p,q}=D_1^{q-N}(A)\star E_{N+1}^{0,N}\star E^p(B)$ for all $p\ge 0$ and all $q\ge N$.
\end{lemma}

  We recall the following standard facts. See \cite{CE} for details.
  
  \begin{lemma}
  \label{Einfinity}
  Let $F$ be a filtered cocomplex and let $E_\bullet$ be the spectral sequence associated with the filtration converging to $H^*(F)$.  
  \begin{enumerate}
\item  If $E_{\infty}^{u,n-u}=0$ for $u<p$, there results an epimorphism $H^n(F)\rightarrow E_{\infty}^{p,n-p}$. 
\item If the filtration is convergent and $E_{\infty}^{u,n-u}=0$ for $u>p$, there results a monomorphism ${E_{\infty}^{p,n-p}\rightarrow H^n(F)}$.
\end{enumerate}
  \end{lemma}
  
  We will apply this result in the cases $p=0$ and $p=n$ to obtain results about $E_{\infty}^{0,n}$ and $E_{\infty}^{n,0}$ in the case where $A$ is a Koszul algebra.

  \begin{prop}
  \label{survivesToE3}
 Let $A$ be a quadratic graded $k$-algebra, $\I\subset A_{\ge 2}$  a graded ideal of $A$, and $B=A/\I$. Then the differential $d_2^{0,1}:E_2^{0,1}\rightarrow E_2^{2,0}$ is an isomorphism in internal degrees greater than 2. If, additionally, $A$ and $B$ are as in Lemma \ref{ENpage}, then for $p\ge 1$ and in internal degrees greater than $p+2$, the image of $d_2^{p,1}$ consists of elements of $E_2^{p+2,0}$ which can be represented as sums in $E^2(B)\star E^p(B)$.
  \end{prop}
  
  \begin{proof}
   Setting $p=0$ and $n=1$ in Lemma \ref{Einfinity}(1) implies that $E^1(A)=E^{1,1}(A)$ maps onto $E_{\infty}^{0,1}$, so $E_{\infty}^{0,1}$ is concentrated in internal degree 1. For $N>2$, the differentials $d_N^{0,1}$ are zero, so cocycles of $d_2^{0,1}$ are permanent, unbounded cocycles and $E_3^{0,1}=E_{\infty}^{0,1}$. Since $E_3^{0,1}\subset E_2^{0,1}\isom \Ext_A^0(\I,k)$ and we assume that $\I$ contains no elements of degree less than 2, $\ker\ d_2^{0,1}=E_3^{0,1}=0$. 
      
  Setting $p=n=2$ in Lemma \ref{Einfinity}(2) implies  there is a monomorphism $E_{\infty}^{2,0}\hookrightarrow E^2(A)=E^{2,2}(A)$.  For $N>2$, the differential $d_N^{2-N,N-1}$ is zero. Since we also have $d_N^{2,0}=0$ for all $N>1$,  $E_3^{2,0}=E_{\infty}^{2,0}$. We conclude that $E_3^{2,0}=E_2^{2,0}/\im\ d_2^{0,1}$ is concentrated in internal degree 2. It follows that $d_2^{0,1}$ is surjective, hence is an isomorphism, in internal degrees greater than 2.
  
 For the second statement, we have $E_2^{p,1}\isom E_2^{0,1}\star E^p(B)$ and $E_2^{p+2,0}\isom E_2^{2,0}\star E^p(B)$. Since $d_2^{0,1}$ is surjective in internal degrees $\ge 2$ and since the spectral sequence differential respects the left $E(B)$-module structure, the result follows.
 
  \end{proof}
  
  We now prove that, as a result of Lemma \ref{lifting2}, we can transform the left action of $E^1(A)$ on $E_{r-1}^{0,r-2}$ into a left action of $E^1(B)$ on $E^{r-1}(B)$ using Proposition \ref{K2matrixCond}.

  \begin{prop}
  \label{survivesForever}
Let $A$ be a graded $k$-algebra and let $B=A/\I$ be a graded factor algebra such that $_A\I$ is a $\K_1$ $A$-module. Assume $B$ acts trivially on $\Ext_A(B,k)$.  If $N> 2$ and $2< r\le N$, then $\im\ d_r^{N-r,r-1}\subset E_2^{1,0}\star E^{N-1}(B)$.
    \end{prop} 

  \begin{proof}
 
 For each $2\le r\le N$, let $\{e_i^r\}$ be a homogeneous basis for $Q^r$. Let $M^r$ be the matrix of $\partial_Q^r$ with respect to these bases. Without loss of generality, we may assume the $\{e_i^r\}$ are chosen such that the nonzero rows of $L(M^r)$ are linearly independent.
%

Let $(a_1, a_2,\ldots, a_r)\in \cS_r^{0,r-1}$ represent a class in $E_r^{0,r-1}$ as in Lemma \ref{staircase}. 
By  Lemma \ref{ENpage}, we may assume that for any basis element $e=e^r_i$ of $Q^r$, $a_r(e)(1)\in I_0^{A_{\ge 2}}$. Since $a_r$ is $B$-linear, the minimality of the resolution $Q^{\bullet}$  implies that $d_ha_r(e)=0$ unless the $i$-th row of $M^r$ contains a linear element. 
By Proposition \ref{K2matrixCond} and our assumption on the nonzero rows of $L(M^r)$, the $i$-th row of $M^r$ contains a linear element if and only if $[(e_i^r)^*]\in E^r(B)$ is in the subalgebra generated by $E^1(B)\star E^{r-1}(B)$. 

Therefore, the image of $d_r^{0,r-1}$ consists of those classes in $E_r^{r,0}$ whose $E_2^{r,0}$ representatives, under the isomorphism of Lemma \ref{triviality} above, are in the subalgebra of $E^r(B)$ generated by $E^1(B)\star E^{r-1}(B)$.

Since $E_2^{1,0}=E_{\infty}^{1,0}$, we have $$E^1(B)\star E^{r-1}(B)\isom E_2^{1,0}\star E^{r-1}(B)=E_r^{1,0}\star E^{r-1}(B)$$ Thus classes in $\im\ d_r^{0,r-1}$ are equal to their $E_2^{r,0}$ representatives and the result holds for $r=N$.
The spectral sequence differential respects the right $E(B)$-module structure on $E_r^{*,r-1}$, so by Lemma \ref{ENpage}, $\im\ d_r^{N-r,r-1}=\im\ d_r^{0,r-1}\star E^{N-r}(B)\subset E_r^{1,0}\star E^{N-1}(B)$ as desired.

  \end{proof}
  
 \begin{rmk}
 \label{survivesRmk}
 A close examination of the proof above and Lemmas \ref{isom} and \ref{ENpage} reveals the slightly stronger fact that if a class in $\Ext_A^r(\I,k)$ is also in the subspace $D_1^{r}(A)\star\Ext^0_A(\I,k)$, then the image of the corresponding class in $E_{r+2}^{0,r+1}$ under $d_{r+2}^{0,r+1}$ is in $E_2^{1,0}\star E^{r+1}(B)$. This fact is often useful in applications. See Section \ref{families}.
 \end{rmk}

We are now able to prove our main theorem.

  \begin{thm} 
  \label{factorThm}
  Let $A$ be a Koszul algebra and $\I\subset A$ an ideal. Assume $B=A/\I$ acts trivially on $\Ext_A(B,k)$.
 If $_A\I$ is a $\K_2$ $A$-module, then $B=A/\I$ is a $\K_2$ algebra.
 \end{thm}
 
\begin{proof}
Since $A$ is quadratic, $_A\I$ is a $\K_1$ $A$-module by Corollary \ref{K2ModsAreKoszul}.
Let $N$ be minimal such that a class $[\z]\in E^N(B)$ is not generated by $E^1(B)$ and $E^2(B)$. Without loss of generality, we may assume $[\z]$ is homogeneous in the bigrading on $E(B)$. Then $N>2$ and the internal degree of $[\z]$ is at least $N+1$. The corresponding class $\a\in E_2^{N,0}\isom E^N(B)$ is a permanent cocycle.

By Proposition \ref{survivesToE3}, $\a\notin\im\ d_2^{N-2,1}$, so $\a$ survives to a nonzero class $[\a]\in E_3^{N,0}$. By Proposition \ref{survivesForever}, $[\a]$ survives to a nonzero class $\a_{\infty}\in E_{\infty}^{N,0}$. By Lemma \ref{Einfinity}, $E_{\infty}^{N,0}\hookrightarrow E^N(A)$, so $E^N(A)$ is not concentrated in internal degree $N$. This contradicts the Koszulity of $A$.

\end{proof}

We remark that if $\I$ is a Koszul $A$-module, then $B$ acts trivially on $\Ext_A(B,k)$ by degree considerations. If we further assume that $\I$ is generated in degree 2, then $B$ is a Koszul algebra. This special case of Theorem \ref{factorThm} was proved by Backelin and Froberg in \cite{BackFro}.

We also note that  Example 9.3 of \cite{CS} shows Theorem \ref{factorThm} is false if $A$ is only assumed to be a $\K_2$ algebra.
The following example shows that the hypothesis that $B$ acts trivially on $\Ext_A(B,k)$ cannot be removed from Theorem \ref{factorThm}.

\begin{example}
Let $A=k\la x,y\ra/\la x^2-xy\ra$, and let $\I=\langle yx\rangle$ be a two-sided ideal. The algebra $A$ is isomorphic to the monomial quadratic algebra $k\la X, Y\ra/\la XY\ra$, hence $A$ is a Koszul algebra (see Corollary 4.3 of \cite{PP}). As a left $A$-module, $_A\I=Ayx+Ayx^2$ and
$$0\rightarrow A(-4)\xrightarrow{\begin{pmatrix} x^2 & -x\\ \end{pmatrix}} A(-2)\oplus A(-3)\xrightarrow{\begin{pmatrix} yx & yx^2\\ \end{pmatrix}^T} 
\I\rightarrow 0$$
is a graded free resolution of $_A\I$. 
The matrix criterion of Proposition \ref{K2matrixCond} shows that $_A\I$ is a $\K_1$ $A$-module. The Hilbert series of $B=A/\I$ is easily seen to be $1+2t+2t^2+t^3/(1-t)$. Therefore, the Poincar\'{e} series of $_Bk$ begins $1-2t+2t^2-t^3-t^4+\cdots$. The negative coefficient of $t^4$ implies $\dim \Ext_B^{i,4}(k,k)\neq 0$ for some $i\neq 4$. Thus the quadratic algebra $B$ is not a Koszul algebra, so $B$ is not a $\K_2$ algebra.
One can check that the image of $x$ in $B$ acts nontrivially on $\Ext_A(B,k)$ by sending the class in $\Ext_A^0(B,k)$ corresponding to the generator $yx$ to the class corresponding to the generator $yx^2$.
\end{example}

\section{Face Rings}
\label{faceRings}

We recall from Section 1 that there is a correspondence which associates to a simplicial complex $\D$ the Stanley-Reisner ideal $I_{\D}$ generated by the monomials $\Pi_{i\in \tau} x_i$ for $\tau\subset [n]$, $\tau\notin\D$. We also recall that the factor algebra $k[\D]=k[x_1,\ldots,x_n]/I_{\D}$ is called the face ring of the simplicial complex $\D$. For simplicity, we assume that $I_{\D}$ contains no linear monomials.
%
%
%
Our main tool is the following special case of Theorem \ref{factorThm}.

\begin{thm}
\label{K2FaceRings}
If $I$ is an ideal in $S=k[x_1,\ldots,x_n]$ and if $I$ is a $\K_2$ $S$-module, then $S/I$ is a $\K_2$ algebra.
\end{thm}

\begin{proof}
Since $S$ is commutative, $S/I$ acts trivially on $\Ext_S(S/I,k)$. 

\end{proof}

Let $\D$ be an abstract simplicial complex on $[n]$. If $\t\in\D$, the \emph{link} of $\t$ in $\D$ is $\lk_{\D}\t =\{\s\in\D\ |\ \s\cap\t=\emptyset, \s\cup\t\in\D\}$. The \emph{Alexander dual} complex of $\D$ is the simplicial complex  $\D^*=\{[n]-\t\ |\ \t\notin\D\}$. We denote the subcomplex of $\D$ whose maximal faces are the $q$-faces of $\D$ by $\D(q)$. If all maximal faces of $\D$ have the same dimension, we say that $\D$ is \emph{pure}.
We adopt the usual terminology and call $\D$  \emph{Cohen-Macaulay over $k$} if $\D$ is pure and if for all faces $\t\in\D$ and all $i<\dim \lk_{\D} \t$ we have $\widetilde{H}_i(\lk_{\D}\t, k)=0$. (We remind the reader that for any simplicial complex $\D$, $\widetilde{H}_i(\D,k)=0$ for $i<0$ unless $\D=\{\emptyset\}$ is the irrelevant complex, in which case $\widetilde{H}_{-1}(\D,k)=k$ and $\widetilde{H}_i(\D,k)=0$ otherwise.) This definition is motivated by the fundamental result of Reisner (see \cite{Reis}) that $\D$ is a Cohen-Macaulay complex over $k$  if and only if $k[\D]$ is a Cohen-Macaulay ring. In \cite{ER}, Eagon and Reiner characterized the Cohen-Macaulay property in terms of resolutions of $I_{\D}$.

\begin{thm}[\cite{ER}]
\label{CMcomplex}
The simplicial complex $\D^*$ is Cohen-Macaulay if and only if the squarefree monomial ideal $I_{\D}$ has a linear free resolution as a $k[x_1,\ldots,x_n]$-module.
\end{thm}

In \cite{Stanley}, Stanley introduced the more general notion of a sequentially Cohen-Macaulay complex. We call $\D$ \emph{sequentially Cohen-Macaulay over $k$} if $\D(q)$ is Cohen-Macaulay for all $q\in\N$. 
In \cite{Duval}, Duval proved that this definition is equivalent to Stanley's.  We note that if $\D$ is pure, then $\D$ is sequentially Cohen-Macaulay if and only if $\D$ is Cohen-Macaulay. Herzog and Hibi proved the analog of Theorem \ref{CMcomplex} for sequentially Cohen-Macaulay complexes.

\begin{thm}[\cite{HH}]
\label{SCMcomplexes}
The complex $\D^*$ is sequentially Cohen-Macaulay if and only if  $I_{\D}$ has a componentwise linear resolution as a $k[x_1,\ldots,x_n]$-module.
\end{thm}

In light of Proposition \ref{CLimpliesK2} we obtain the following sufficient condition for $k[\D]$ to be a $\K_2$ algebra.

\begin{cor}
\label{SCMimpliesK2}
If $\Delta^*$ is sequentially Cohen-Macaulay over $k$, then $k[\D]$ is a $\K_2$ algebra.
\end{cor}

\begin{proof}
The statement follows from Theorem \ref{SCMcomplexes}, Proposition \ref{CLimpliesK2}, and Theorem \ref{K2FaceRings}.

\end{proof}

Example \ref{K2notCL} below shows the sufficient condition of Corollary \ref{SCMimpliesK2} is not necessary. For examples where $k[\D]$ is a $\K_2$ algebra and $I_{\D}$ is not even a $\K_2$ module, see Example \ref{TowerCounter} and the algebra $k[\D_7]$ of Section \ref{families}.

A simplicial complex $\D$ is called \emph{Buchsbaum over $k$} if $\D$ is pure and if for all non-empty faces $\t\in\D$ and all $i<\dim \lk_{\D} \t$ we have $\widetilde{H}_i(\lk_{\D}\t, k)=0$. Thus $\D$ is Buchsbaum over $k$ if $\D$ is Cohen-Macaulay over $k$.  In Section \ref{families}, the dual complex $\D_7^*$ is pure, but the link of vertex $\{d\}$ in $\D_7^*$ is two-dimensional and disconnected, so $\D_7^*$ is not Buchsbaum over any $k$. 

Nonetheless, the statement of Corollary \ref{SCMimpliesK2} is false if ``sequentially Cohen-Macaulay'' is replaced by ``Buchsbaum over $k$.''  One counterexample is the face ring $k[\D'_6]$ discussed in Section \ref{families}.
 
Since Cohen-Macaulay and sequentially Cohen-Macaulay  complexes can be characterized in terms of their combinatorial topology, we hope to establish geometric criteria on $\D^*$ under which $I_{\Delta}$ is a $\K_2$ module. 

%

\section{Examples}
\label{examples}

For $A$ a graded algebra, we use the shorthand $A(d_1^{i_1}, d_2^{i_2}, \ldots, d_n^{i_n})$ to denote the graded free $A$-module $A(d_1)^{\oplus i_1}\oplus A(d_2)^{\oplus i_2}\oplus \cdots \oplus A(d_n)^{\oplus i_n}$. 

The first example shows that submodules of $\K_2$ modules need not be $\K_2$. Thus we do not expect an analog of Lemma 3.2(3) for $\K_1$ or $\K_2$ modules. The calculation is also used in Example \ref{TowerCounter} below.

\begin{example}
\label{submodule}
Let $S=k[a,b,c,d,e,f]$ and  $I=\la abc, cde, ae\ra$. The complex $\D^*$ is shown in Figure \ref{submodulefig}. It is easy to check that $\D^*(q)$ is Cohen-Macaulay for $q=0,1,2,3$. By  Theorem 2.1 of \cite{HH}, $I$  has a componentwise linear $S$-module resolution, hence is a $\K_2$ $S$-module by Proposition \ref{CLimpliesK2}. The subideal $J=\la abc, cde\ra$ is not a $\K_2$ $S$-module. Indeed the following complex is a minimal projective $S$-module resolution of $_SJ$.

$$0\rightarrow S(-5)\xrightarrow{\begin{pmatrix} de & -ab\\ \end{pmatrix}}S(-3,-3)\xrightarrow{\begin{pmatrix} abc\\ cde\\ \end{pmatrix}} J\rightarrow 0$$
We see that the matrix criterion of Lemma \ref{K2matrix} is not satisfied by this resolution.

\begin{figure}[ht] 
   \centering
   \includegraphics[width=2in]{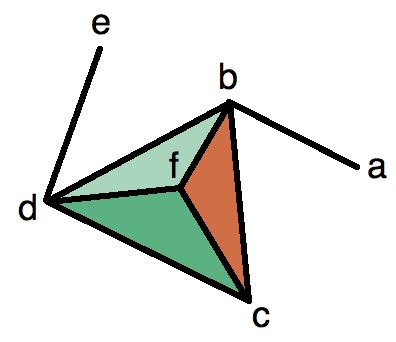} 
   \caption{The complex $\D^*$ for $I$ in Example \ref{submodule}.}
         \label{submodulefig}
\end{figure}
\end{example}

Next, we give an example  of a face ring which is $\K_2$ but not sequentially Cohen-Macaulay.
 
\begin{example}
\label{K2notCL}
Let $S=k[a,b,c,d,e,f]$ and let $I=\la abc, def, abef\ra$. The complex $\D^*$ and the subcomplex $\D^*(2)$ are shown in Figure \ref{K2notCLfig}. Since $\D^*(2)$ is two-dimensional and disconnected, it is not Cohen-Macaulay.
By Theorem \ref{SCMcomplexes}, $I$ does not have a componentwise linear $S$-module resolution. The following complex is a minimal projective $S$-module resolution of $_SI$.
$$0\rightarrow S(-5,-5)\xrightarrow{\begin{pmatrix}  ef & 0 &-c\\  0 & ab & -d\\ \end{pmatrix}}S(-3,-3,-4)\xrightarrow{\begin{pmatrix} abc\\def\\abef\\ \end{pmatrix}} I\rightarrow 0$$

\begin{figure}[ht] 
   \centering
   \includegraphics[width=3in]{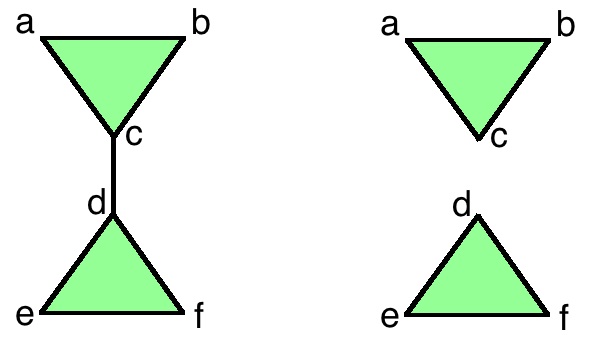} 
   \caption{The complex $\D^*$ for $I$ in Example \ref{K2notCL}  and $\D^*(2)$.}
      \label{K2notCLfig}
\end{figure}
By Proposition \ref{K2matrixCond}, $_SI$ is a $\K_2$ $S$-module. Thus $S/I$ is a $\K_2$ algebra by Theorem \ref{K2FaceRings}.
The subideal $J=\la abc, abef\ra$ is easily seen to be a $\K_2$ $S$-module. However, the factor module $I/J$ is not. The complex
$$0\rightarrow S(-5)\xrightarrow{\begin{pmatrix}ab\end{pmatrix}}S(-3)\xrightarrow{\begin{pmatrix}def\\ \end{pmatrix}}I/J\rightarrow 0$$
is a minimal free $S$-module resolution of $I/J$, and it  fails the matrix condition of Lemma \ref{K2matrix}.
\end{example}

Our final example in this section shows that a face ring can be $\K_2$ even if the Stanley-Reisner ideal $I_{\D}$ is not a $\K_2$ $S$-module. This provides a counterexample to the converse of Theorem \ref{factorThm}.

\begin{example}
\label{TowerCounter}
Let $S=k[a,b,c,d,e]$ and let $I=\la abc, cde, abde\ra$. Let $A=S/I$ and let $B=A/\la c\ra$.
 Observe that  $B\cong k[a,b,d,e]/\la abde\ra$.  By Corollary 9.2 of \cite{CS}, $B$ is a $\K_2$ algebra. The following sequence is the beginning of a minimal projective $A$-module resolution of $_AB$. From this description it is clear how to continue the resolution.
$$\cdots\xrightarrow{\begin{pmatrix} \alpha  & 0 & 0\\  0 & \alpha & 0\\ 0 & 0 & \gamma\\ 0 & 0 & \beta\\ \end{pmatrix}} A(-4^2, -5^2)\xrightarrow{\begin{pmatrix} \gamma\\ \beta' \\ \end{pmatrix}}A(-3^2)\xrightarrow{\begin{matrix}\alpha\\ \end{matrix}}A(-1) \xrightarrow{\begin{pmatrix} c\\ \end{pmatrix}} A\rightarrow B$$
where
$$\alpha = \begin{pmatrix} ab\\ de\\ \end{pmatrix}\quad \beta=\begin{pmatrix} ab & 0\\ 0 & de\\ \end{pmatrix}\quad \beta'=\begin{pmatrix} de & 0\\ 0 & ab\\ \end{pmatrix}\quad \gamma = \begin{pmatrix} c & 0\\ 0 & c\\ \end{pmatrix}$$
 By the matrix criterion of Lemma \ref{K2matrix}, $B$ is a $\K_2$ $A$-module. By Theorem 7.4 of \cite{CS}, $A$ is a $\K_2$ algebra.

The following is a minimal projective $S$-module resolution of the defining ideal $I$ of $A$.
$$0\rightarrow S(-5^2)\xrightarrow{\begin{pmatrix} de & 0 & -c\\ 0 & ab &-c\\ \end{pmatrix}} S(-3^2,-4)\xrightarrow{\begin{pmatrix} abc\\ cde\\ abde\\ \end{pmatrix}} I\rightarrow 0$$

As the linearization of the matrix on the left has dependent rows, $I$ is not a $\K_2$ $S$-module. Thus the converse of Theorem \ref{factorThm} is false. 

The algebra $A$ is a noteworthy example in another regard. Let $J$ be an ideal in a commutative Koszul algebra $S$. Suppose $J$ has a componentwise linear $S$-module resolution. Then by Theorem \ref{factorThm}, $S/J$ is a $\K_2$ algebra. The subideal $J_{\la 0, d\ra}$ has a componentwise linear resolution for all $d\ge 0$, so $S/J_{\la 0,d\ra}$ is a  $\K_2$ algebra for all $d\ge 0$.  Cassidy and Shelton asked if such a ``tower'' theorem holds more generally for $\K_2$ algebras. The algebra $A$ resolves that question negatively, as we now show.

Let $C=k[a,b,c,d,e]/\la abc, cde\ra$ be the algebra obtained by discarding the highest degree generator of $I$.

\begin{prop}
The algebra $C$ is not $\K_2$.
\end{prop}

\begin{proof}
By the usual identifications of $E^1(C)$ and $E^2(C)$ respectively with the $k$-linear duals of the vector spaces of generators and minimal defining relations of $C$, we have $E^1(C)=\Ext_C^{1,1}(k,k)$ and $E^2(C)=\Ext_C^{2,2}(k,k)\oplus\Ext_C^{2,3}(k,k)$. 
We prove the proposition by showing that $\Ext^{3,5}_C(k,k)\neq 0$.

The Hilbert series of the defining ideal $J=\la abc, cde\ra$ is easily seen to be $2t^3+10t^4+29t^5+\cdots$ thus the Hilbert series of $C$ is $$\dfrac{1}{(1-t)^5}-(2t^3+10t^4+29t^5+\cdots)=1+5t+15t^2+33t^3+60t^4+97t^5+\cdots$$
Inverting the formal power series, we find the Poincare series of $C$ is
$$1-5t+10t^2-8t^3-5t^4+18t^5+\cdots$$
Since $J_2=0$, the diagonal subalgebra $\bigoplus \Ext_C^{i,i}(k,k)$ of $E(C)$ is isomorphic to $E(S)$, which is the exterior algebra on a 5-dimensional vector space (see Proposition 3.1 of \cite{PP}). In particular, $\dim\Ext_C^{2,2}(k,k)=10$ and $\dim \Ext_C^{5,5}(k,k)=1$. 
From the minimal $S$-module resolution of $J$ given in Example \ref{submodule}, we see that $$\Ext_S^1(C,k)\isom \Ext_S^0(J,k)=k(-3)\oplus k(-3)$$  
$$\Ext_S^2(C,k)\isom\Ext_S^1(J,k)=k(-5)$$ and $\Ext_S^q(C,k)=0$ for $q>2$.

Consider the spectral sequence $\Ext_C^p(k,\Ext_S^q(C,k))\Rightarrow \Ext_S^{p+q}(k,k)$. By Lemma \ref{triviality}, $E_2^{p,q}\isom E^p(C)\otimes \Ext_S^q(C,k)$. 
Every element in $\Ext_C^{4,5}(k,k)$ represents a permanent cocycle in $E_2^{4,0}$. 
Since the target of the spectral sequence is the Yoneda algebra of a Koszul algebra, each of these cocycles must be eventually bounded. The differentials which could bound these cocycles are
$$E_2^{2,1}\rightarrow E_2^{4,0}\qquad E_3^{1,2}\rightarrow E_3^{4,0}\quad \text{and}\quad E_4^{0,3}\rightarrow E_4^{4,0}$$
Since $\Ext_S^3(C,k)=0$, $E_4^{0,3}=0$. Since $E_3^{1,2}\subset E_2^{1,2}\isom E^1(C)\otimes \Ext_S^2(C,k)$, the vector space $E_3^{1,2}$ is concentrated in internal degree 6, so it cannot bound a cocycle with internal degree 5. Thus $E_2^{4,0}\isom E_2^{2,1}\isom E^2(C)\otimes \Ext_S^1(C,k)$ in internal degree 5. From the calculations above, we have $\dim \Ext_C^{4,5}(k,k) = 20$. Since $\dim \Ext_C^{5,5}(k,k)=1$, it follows from the Poincare series that $\dim \Ext_C^{3,5}(k,k)=1$.

\end{proof}

 \end{example}

\begin{rmk}
If $A=T(V)/R$ is a graded $k$-algebra as in Section \ref{factorK2}, the \emph{quadratic part} of $A$ is the algebra $qA=T(V)/I$ where $I$ is the ideal generated by $R\cap T^2(V)$. A weaker version of the Cassidy-Shelton question is: Is $qA$ is a Koszul algebra whenever $A$ is a $\K_2$ algebra? The answer to this question is also no, as demonstrated by the algebra $A=k\la a,b,c,d,e,f,g,h,k\ra/\la ag, be-gh, cd-ef, dk, abc\ra$. We leave the verification of this counterexample as a straightforward exercise in applying Lemma \ref{K2matrix}.

\end{rmk}

\bigskip

\section{Two families of face rings}
\label{families}

In this section we discuss the $\K_2$ property for two families of face rings.

Fix $n\ge 3$ and let $S_n=k[x_1,\ldots,x_n]$. Let $G_n= \{ x_ix_{i+1}x_{i+2}\ |\ i=1,\ldots n-2\}$.  Let $I_{\D_n}=\la G_n\ra$ and $I_{\D'_n}=\la G_n\cup \{x_1x_2x_n, x_1x_{n-1}x_n\}\ra$.  Let $k[\D_n]=S_n/I_{\D_n}$ and $k[\D'_n]=S_n/I_{\D'_n}$.

For $3\le n<6$, $\D_n^*$ and $(\D'_n)^*$ are readily seen to be Cohen-Macaulay over any field $k$, hence $k[\D_n]$ and $k[\D'_n]$ are $\K_2$ algebras for $3\le n< 6$ by Theorem \ref{SCMimpliesK2}. The complex $\D_6^*$ is also easily seen to be Cohen-Macaulay over any field $k$, so $k[\D_6]$ is $\K_2$ as well.

The complex $(\D'_6)^*$ was considered in \cite{Hibi} (Example 26.2). The link of each vertex in $(\D'_6)^*$ is contractible, thus $(\D_6')^*$ is Buchsbaum over any field. A computer calculation using the program Macaulay2 shows  $\dim E^{4,6}(k[\D_6'])=37$.  By the usual identification of $E^2(k[\D_6'])$  with the $k$-linear dual of the vector space of minimal defining relations of the algebra $k[\D_6']$, we have $\dim E^{2,3}(k[\D_6'])=6$ and $E^{2,4}(k[\D_6'])=0$. It is therefore impossible to generate $E^{4,6}(k[\D_6'])$ algebraically using $E^1(k[\D_6'])=E^{1,1}(k[\D_6'])$ and $E^2(k[\D_6'])$, so $k[\D_6']$ is not $\K_2$. 

For the remainder of this section we consider the case $n=7$. For convenience, we change notation and let $A=k[a,b,c,d,e,f,g]$ and $B=A/\la abc, bcd, cde, def \ra$. Let $\D=\D_7$. Tensor products of $\mathcal K_2$ algebras are $\mathcal K_2$ (see \cite{CS}), so $B$ is a $\mathcal K_2$ algebra. Let $J=\la abc, bcd, cde, def, efg \ra$ and $C=A/J$ so $C\isom k[\D]$.

\begin{lemma}
\label{bettiNos}
As a left $A$-module, $J$ does not have a linear free resolution. More precisely,
$$\dim \Ext_A^{i,j}(C,k)=\begin{cases} 1 & (i,j)=(0,0)\\ 5 & (i,j)=(1,3)\\ 4 &(i,j)=(2,4)\\ 1 & (i,j)=(2,6)\\ 1 & (i,j)=(3,7)\\ 0 & else\\ \end{cases}$$
\end{lemma}

\begin{proof}
Clearly, $\Ext_A^0(C, k)=\Ext_A^{0,0}(C,k)=k$. By the Eagon-Reiner Theorem \cite{ER}, for $i\ge 1$ we have
$$\dim \Ext_A^{i,j}(C,k) = \sum_{\substack{\sigma\in \D^*\\ |\sigma|=n-j\\ }} \dim \widetilde{H}_{i-2}(\lk_{\D^*}(\sigma), k)$$
The facets of $\D^*$ are $abcd, abcg, abfg, aefg, defg$. The link of a simplex is the irrelevant complex if and only if the simplex is a facet, so $\Ext_A^1(C,k)=\Ext_A^{1,3}(C,k)\isom k^5$. The link complexes which are disconnected are $\lk_{\D^*}(\sigma)$ for $\sigma\in\{ abc, abg, afg, efg, d\}$. The realization of each of these complexes is homotopic to two points. This verifies the dimensions of the graded components of  $\Ext_A^2(C,k)$. The realizations of all other link complexes are contractible, except $\lk_{\D^*} \{\emptyset\}=\D^*\simeq S^1$. Thus $\Ext_A^3(C,k)=\Ext_A^{3,7}(C,k)=k$ and $\Ext_A^i(C,k)=0$ for $i>3$.

\end{proof}

\begin{lemma}
\label{vanish}
$E^{3,6}(C)=E^{4,7}(C)=0$
\end{lemma}

\begin{proof}
Consider the change-of-rings spectral sequence for $B\rightarrow C$. The kernel of this map is $K=\la efg\ra$. The following is the beginning of a minimal graded projective resolution for $_BK$.
$$\cdots B(-6^3)\xrightarrow{\begin{pmatrix} bc\\ ce\\ ef\\ \end{pmatrix}} B(-4)\xrightarrow{\begin{pmatrix}d\\ \end{pmatrix}} B(-3) \xrightarrow{\begin{pmatrix} efg\\ \end{pmatrix}} K\rightarrow 0$$
Evidently, for $i<4$, 
$$\dim \text{Ext}_B^{i,j}(C,k)=\begin{cases} 1 & (i,j)=(0,0)\\ 1 & (i,j)=(1,3)\\ 1 &(i,j)=(2,4)\\  3 & (i,j)=(3,6)\\  0 & \text{else}\\ \end{cases}$$
It follows  that the differentials
$$E_2^{1,1}\rightarrow E_2^{3,0}\qquad E_3^{0,2}\rightarrow E_3^{3,0}$$
are both zero in internal degree 6 and the differentials
$$E_2^{2,1}\rightarrow E_2^{4,0}\qquad E_3^{1,2}\rightarrow E_3^{4,0}\qquad E_4^{0,3}\rightarrow E_4^{4,0}$$
are all zero in internal degree 7. 

The differentials $E_N^{3-N,N-1}\rightarrow E_N^{3,0}$ are all zero for $N>3$, thus $E^{3,6}(C)\cong (E_2^{3,0})_6 = (E_{\infty}^{3,0})_6$. Since $B$ is a $\K_2$ algebra  with defining relations of degrees 2 and 3, $E^{3,6}(B)=0$. Thus $E^{3,6}(C)=0$.

The differentials $E_N^{4-N,N-1}\rightarrow E_N^{4,0}$ are all zero for $N>4$, thus $E^{4,7}(C)\cong (E_2^{4,0})_7 = (E_{\infty}^{4,0})_7$. Again, since $B$ is a $\K_2$ algebra  with defining relations of degrees 2 and 3,  $E^{4,7}(B)=0$. Thus $E^{4,7}(C)=0$.

\end{proof}

\begin{lemma}
The change-of-rings spectral sequence for $A\rightarrow C$ collapses at the $E_4$ page.
\end{lemma}

\begin{proof}
By Lemma \ref{vanish}, $E^{4,7}(C)=0$. Thus the differential $E_4^{0,3}\rightarrow E_4^{4,0}$ is zero by Lemma \ref{bettiNos}. Since the spectral sequence differential respects the right $E(C)$ module structure, the differential $E_4^{p,3}\rightarrow E_4^{p+4,0}$ is zero for all $p\ge 0$. The differential $E_4^{p,q}\rightarrow E_4^{p+4,q-3}$ is clearly zero if $q<3$ and for $q>3$ it is zero by Lemma \ref{bettiNos}.

\end{proof}

\begin{lemma}
The algebra  $C$ is a $\mathcal K_2$ algebra.
\end{lemma}

\begin{proof}
We continue to work with the change-of-rings spectral sequence for $A\rightarrow C$.
Let $N$ be the smallest cohomology degree such that $E^N(C)\neq D_2^N(C)$ and let $\alpha\in E_2^{N,0}$ be a nonzero class corresponding to a class in $E^N(C)-D_2^N(C)$. Proposition 4.12 implies $\alpha$ survives to the $E_3$ page. Thus $\alpha$ survives to the $E_4$ page or is bounded by $E_3^{N-3,2}$. 

By the corollary to Lemma \ref{vanish}, if $\alpha$ survives to $E_4$, then $\alpha$ survives forever. Since the target of the spectral sequence is the Yoneda algebra of the Koszul algebra $A$ and since the internal degree of $\alpha$ is greater than $N$, we must have $\alpha=0$, a contradiction.

On the other hand, suppose $\alpha$ is bounded by $E_3^{N-3,2}$. By Lemma \ref{bettiNos}, there exist vector spaces $V$ and $W$ (concentrated in degree 0) such that $E_3^{0,2}=V(-4)\oplus W(-6)$. By Lemma \ref{vanish}, $E^{3,6}(C)=0$, so $d_3^{N-3,2}(W(-6)\star E^{N-3}(C))=0$. Arguing as in the proof of Proposition \ref{survivesForever} and Remark \ref{survivesRmk}, $d_3^{N-3,2}(V(-4)\star E^{N-3}(C))\subset E_2^{1,0}\star E^{N-1}(C)$.  Thus if $\alpha$ is bounded on the $E_3$ page, $\alpha$ corresponds to a class in $D_2^N(C)$ by the minimality of $N$, a contradiction.

\end{proof}

If $L=\la afg, abg\ra\subset C$, then $C/L\isom k[\D'_7]$. One can show $C/L$ is a $\K_2$ algebra by arguing exactly as in the Lemmas above, though verifying $\Ext_C^3(C/L,k)=\Ext_C^{3,6}(C/L,k)$ is more difficult.    We do not know if the face rings in either family with more than 8 generators are $\K_2$.

 \bibliographystyle{amsplain}
\bibliography{bibliog}

\end{document}